\documentclass[12pt,leqno,twoside]{amsart}
\usepackage[T1]{fontenc}
\usepackage[utf8]{inputenc}
\usepackage[babel=true]{csquotes}  
\usepackage{amstext,amsmath,amscd, bezier,indentfirst,amsthm,amsgen,enumerate}
\usepackage[left=2.5cm, right=2.5cm, top=2.5cm,bottom=2.5cm]{geometry} 
\usepackage{todonotes}
\usepackage{subcaption}
\usepackage[all,knot,arc,import,poly]{xy}
\usepackage{amsfonts,color}
\usepackage{amssymb}
\usepackage{latexsym}
\usepackage{epsfig}
\usepackage{graphicx}
\usepackage{srcltx}
\usepackage{enumitem}
\usepackage{accents}
\usepackage{imakeidx}
\usepackage{rotating}
\usepackage{verbatim}
\usepackage{fancyhdr}
\fancyheadoffset[RE,LO]{-0.01\textwidth}
\makeindex
\usepackage{tikz}
\usetikzlibrary{patterns}
\usetikzlibrary{matrix,arrows,decorations.pathmorphing}
\usepackage{tikz-cd}
\tikzset{commutative diagrams/.cd}
\usepackage{framed,lipsum}
\usepackage{float}
\usepackage{pgfplots}
\usetikzlibrary{patterns}
\usetikzlibrary{intersections}
\usetikzlibrary{positioning,shadows,arrows}
\tikzstyle{every node}=[anchor=west, minimum height=3em]
\usetikzlibrary{shapes,snakes}
\definecolor{forestgreen}{rgb}{0.00, 0.39, 0.00} 
\definecolor{blueblue}{rgb}{0.40, 0.00, 1.00} 
\definecolor{sienna}{rgb}{0.33, 0.08, 0.11}

\newcommand{\defi}{\textbf}
\theoremstyle{plain}
\newtheorem{theorem}{Theorem}[section] 

\newtheorem*{theorem*}{Theorem}
\newtheorem*{theoremfr*}{Théorème}
\newtheorem*{hypothesisfr*}{Hypothèse}
\newtheorem*{hypothesis*}{Hypothesis}
\newtheorem{lemma}[theorem]{Lemma}
\newtheorem{proposition}[theorem]{Proposition}
\newtheorem{corollary}[theorem]{Corollary}
\newtheorem{notation}[theorem]{Notation}

\theoremstyle{definition}
\newtheorem{definition}[theorem]{Definition}
\newtheorem*{definition*}{Definition}
\theoremstyle{remark}
\newtheorem{remark}[theorem]{\sc Remark}
\newtheorem*{question*}{\sc Question}
\newtheorem*{remark*}{\sc Remark}
\newtheorem*{remarkfr*}{\sc Remarque}
\newtheorem*{examplefr*}{\sc Exemple}
\newtheorem*{example*}{\sc Example}

\newtheorem{example}[theorem]{\sc Example}

\def\bR{{\mathbb R}}

\def\const.{{\rm const.}}

\def\Int{{\rm Int\ }}

\usepackage[style=alphabetic,natbib=true,backref=true,backend=bibtex]{biblatex}
\usepackage{mathscinet}
\usepackage[pdftex, pdfusetitle, plainpages=false,  bookmarks, bookmarksnumbered,colorlinks, linkcolor=blue, citecolor=red,filecolor=black, urlcolor=black]{hyperref}
\usepackage{pgfplots}
\usepackage{wrapfig}
\pgfplotsset{width=7cm,compat=1.8}
\usepackage[tight]{minitoc} 
\usepackage{caption}
\makeatletter
\renewcommand*{\numberline}[1]{\hb@xt@1em{#1\hfil}} 
\makeatother

\usepackage[style=alphabetic,natbib=true,backref=true,backend=bibtex]{biblatex}
\usepackage{mathscinet}

\keywords{generic projection, real algebraic curve, Poincaré-Reeb tree, permutation, snake, polar curve, bitangent, inflection, dual curve, strict local minimum}
\usepackage{pgfplots}

\pgfplotsset{width=7cm,compat=1.8}

\usepackage{caption}

\begin{document}
\title{Permutations encoding the local shape of level curves of real polynomials via generic projections}
\date{\today}
\author{Miruna-\c Stefana Sorea}
\address{Max-Planck-Institut für Mathematik in den Naturwissenschaften, Leipzig, Germany}
\email{miruna.sorea@mis.mpg.de}
\maketitle

\section*{Abstract}\label{sect:CrestsValleys}
The non-convexity of a smooth and compact connected component of a real algebraic plane curve can be measured by a combinatorial object called the Poincaré-Reeb tree associated to the curve and to a direction of projection. In this paper we show that if the chosen projection avoids the bitangents and the inflectional tangencies to the small enough level curves of a real bivariate polynomial function near a strict local minimum at the origin, then the asymptotic Poincaré-Reeb tree becomes a complete binary tree and its vertices become endowed with a total order relation. Such a projection direction is called generic. We prove that for any such asymptotic family of level curves, there are finitely many intervals on the real projective line outside of which all the directions are generic with respect to all the curves in the family. If the choice of the direction of projection is generic, then the local shape of the curves can be encoded in terms of alternating permutations, that we call snakes. The snakes offer an effective description of the local geometry and topology, well-suited for further computations. 
\section{Introduction}
Let $f:\mathbb{R}^2\rightarrow\mathbb{R}$ be a polynomial function that vanishes at the origin $O$ and that has a strict local minimum at this point. Consider the set $f^{-1}([0,\varepsilon])$ and denote the connected component that contains the origin by $\mathcal{D}_\varepsilon$. For sufficiently small $\varepsilon>0,$ $\mathcal{D}_\varepsilon$ is a topological disk bounded by a smooth Jordan curve, denoted by $\mathcal{C}_\varepsilon$ (see \cite[Lemma 5.3]{So3}). Such a curve is called \emph{the real Milnor fibre} of the polynomial $f$ at the origin.

\subsection{Setting and existing results}
Whenever the origin is a Morse strict local minimum (that is, the Hessian matrix at $O$ is nonsingular; for example, $f(x,y)=x^2+y^2$), and $\varepsilon>0$ is small enough, the topological disk $\mathcal{D}_\varepsilon$ is convex (\cite[Theorem 2.3]{So3}). In this case, by abuse of language we say that $\mathcal{C}_\varepsilon$ is a convex curve. Note that the levels \emph{far} from the origin may be non-convex, as we show in Figure \ref{fig:convex} below. However, throughout this paper we are only interested in the behaviour of the \emph{small enough} level curves near a strict local minimum.

{\centering\vspace{10pt}
\includegraphics[scale=0.5,trim=120 10 100 30, clip]{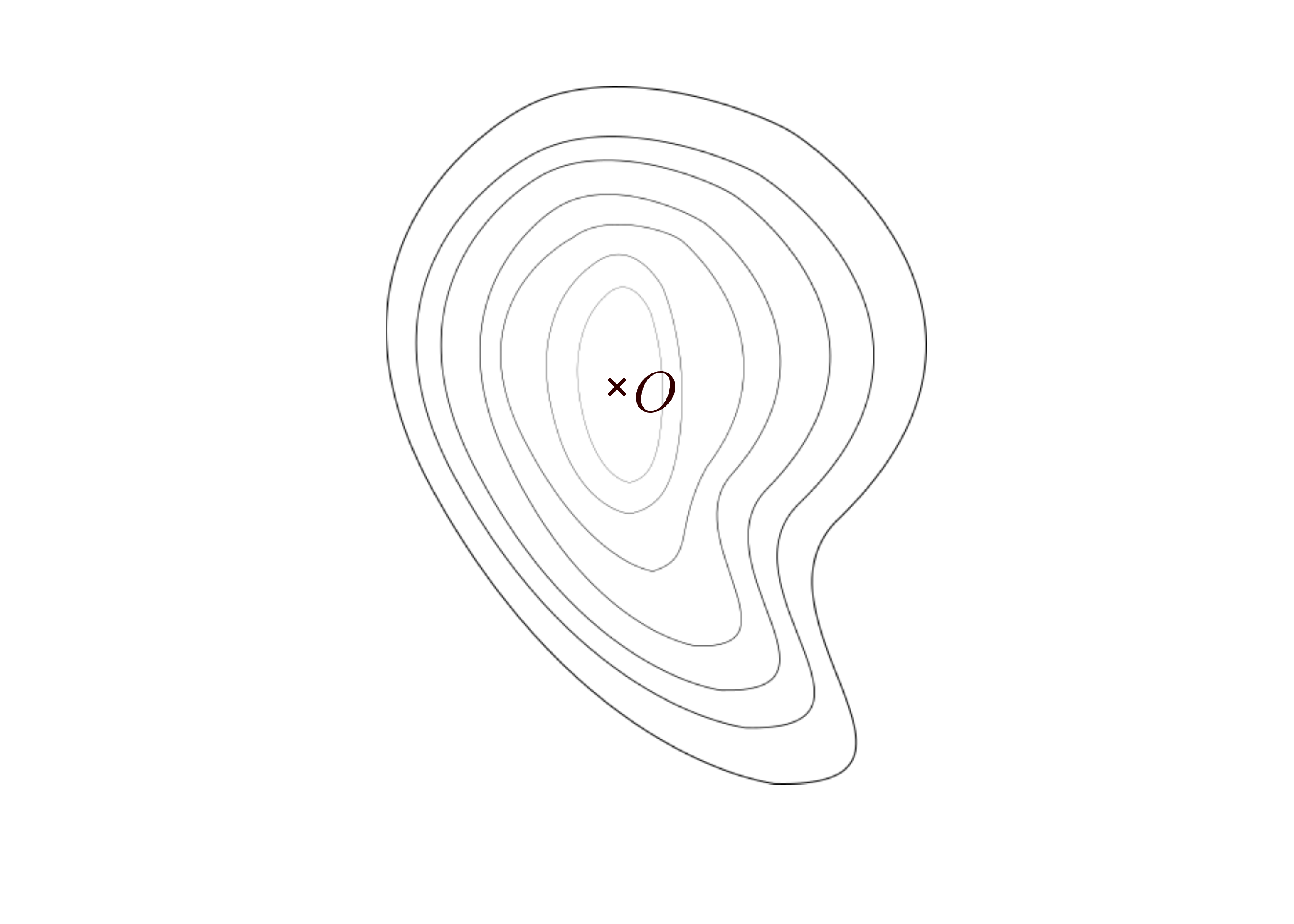} 
\captionof{figure}{Small enough level curves near a Morse strict local minimum become boundaries of convex topological disks.\label{fig:convex}}}
\vspace{10pt}

If the strict local minimum is non-Morse, the small enough level curves may be non-convex. Such an example was given by Coste (see  \cite{So1}, \cite[Example 3.1]{So3}), and some of its level curves near the origin are shown in Figure \ref{fig:contraexCoste} below. Its graph is the surface represented in Figure \ref{fig:contraexCoste3d}.

{\centering\vspace{10pt}
\includegraphics[scale=0.4]{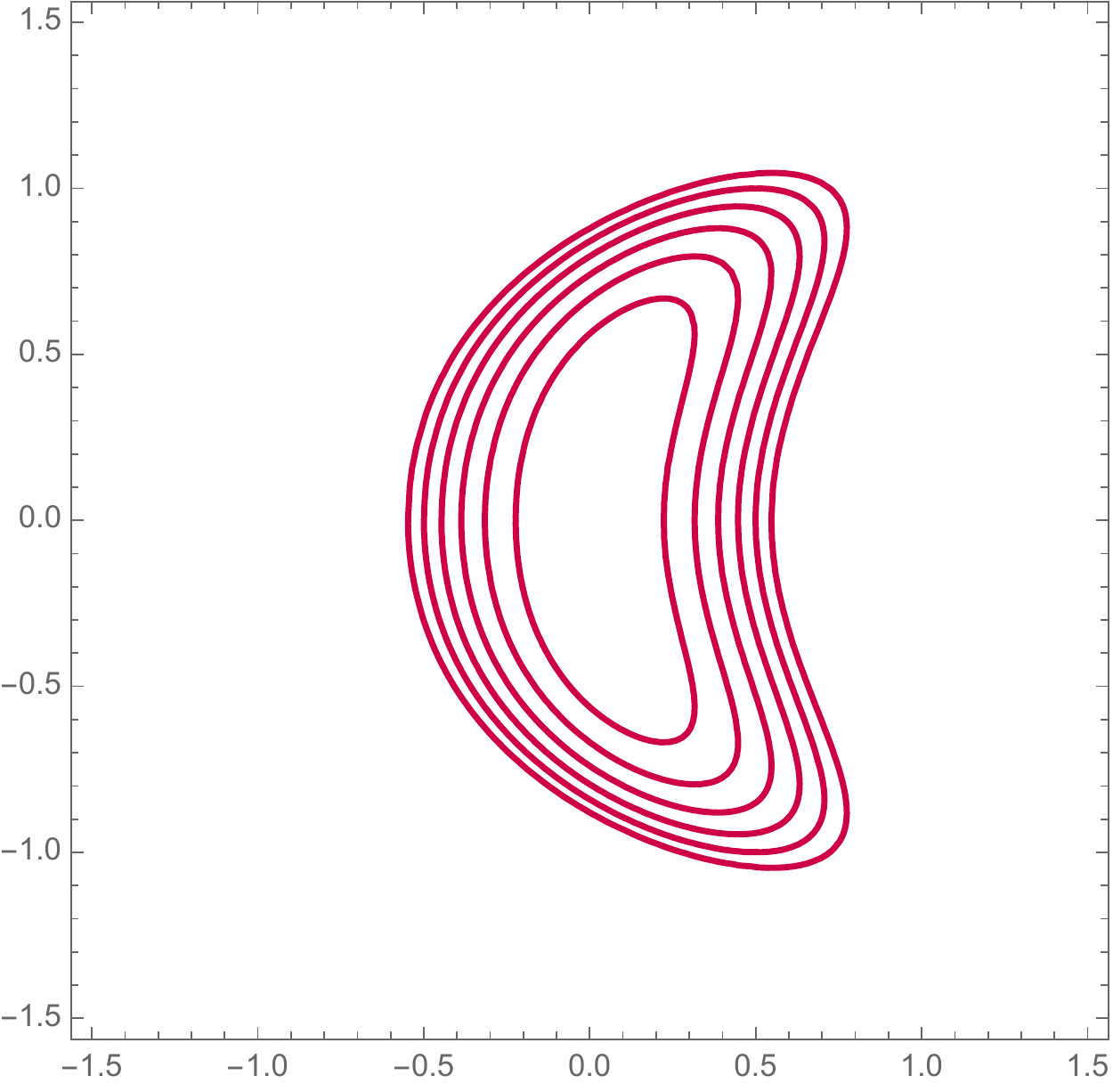} 
\captionof{figure}{Some level curves of Coste's example: $f(x,y)=x^2+(y^2-x)^2$ near the origin.\label{fig:contraexCoste}}}
 \vspace{10pt}

{\centering
\includegraphics[scale=0.17,trim=100 250 100 0, clip]{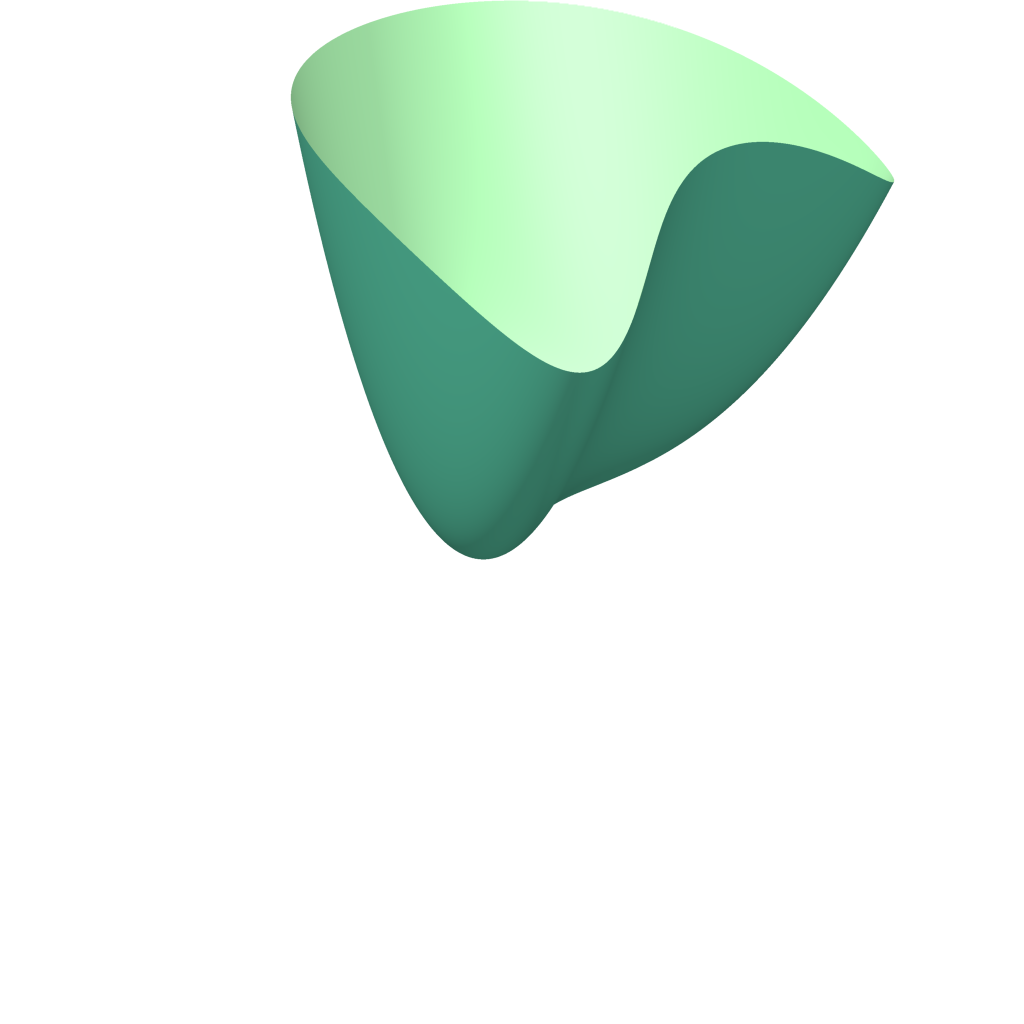} 
\captionof{figure}{The surface $z=x^2+(y^2-x)^2.$\label{fig:contraexCoste3d}}}  
\vspace{10pt}

These phenomena of non-convexity near a strict local minimum were recently studied in \cite[Subsection 4.1]{So3} (see also \cite{So1}). There the author introduced a combinatorial object that encodes the shape of a smooth and compact connected component of a real algebraic plane curve. Its role is to measure how far from being convex a given curve is. This object is inspired from Morse theory and it is called the Poincaré-Reeb tree associated to the curve and to a projection direction, say the vertical direction (see Figure \ref{fig:11Coste} and Figure \ref{fig:1Coste}).

{\centering
\includegraphics[scale=0.2]{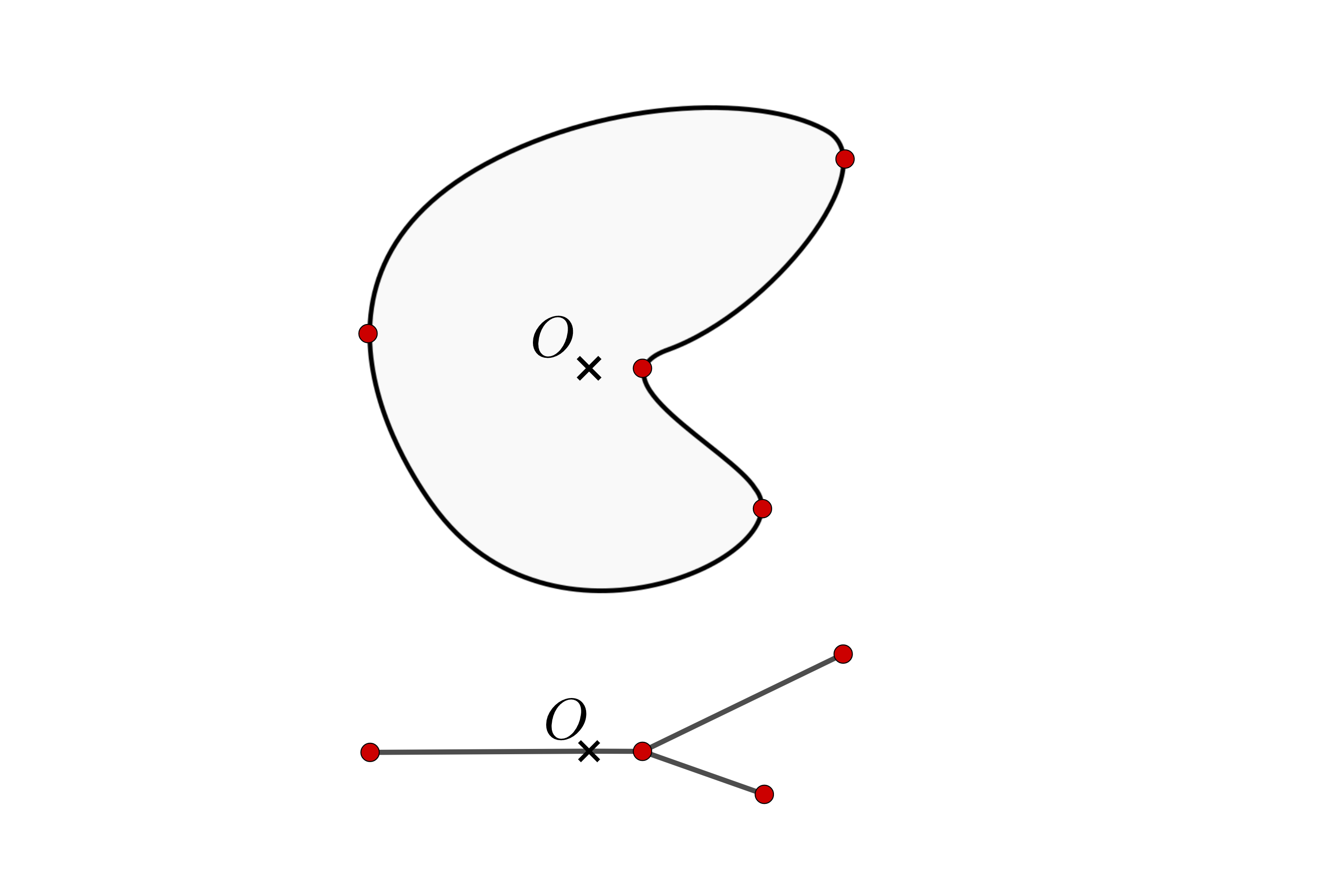} 
\captionof{figure}{A non-convex level curve near a non-Morse strict local minimum at the origin; the Poincaré-Reeb tree associated to it and to the vertical projection direction.\label{fig:11Coste}}}
\vspace{10pt}

An asymptotic study on the Poincaré-Reeb trees of the level curves $\mathcal{C}_\varepsilon$, for $\varepsilon>0$ small enough, near an isolated minimum at the origin was introduced in \cite{So3}. Using the real polar curve, in \cite[Theorem 5.31]{So3} it was shown that the Poincaré-Reeb tree combinatorially stabilises when $\varepsilon>0$ is sufficiently small. This is what we call the asymptotic Poincaré-Reeb tree (\cite[Definition 5.30]{So3}). In addition, its edges do not move backwards towards the origin. Hence no local spiralling phenomena occur (\cite[Theorem 5.31]{So3}).   

For some standard vocabulary related to graphs and trees needed in this paper, the reader may refer to \cite[Subsection 2.1]{So2}.

\subsection{Contributions}
This paper focuses on \emph{the genericity} of the choice of projection direction that is used in the construction of the Poincaré-Reeb tree, and on some combinatorial consequences provided by this genericity assumption. 

By definition, a direction is generic with respect to a curve if it avoids bitangents and inflectional tangencies to the given curve (see Definition \ref{DefGenericDirection}).  If the choice of the direction of projection is generic, then we obtain what we call a generic asymptotic Poincaré-Reeb tree. We prove that besides the properties of an asymptotic Poincaré-Reeb tree, the generic asymptotic Poincaré-Reeb tree acquires special characteristics (see Theorem \ref{mainTh}).

\begin{theorem}[Subsection \ref{subs:genAsPRT}]\label{mainThIntro}
A generic asymptotic Poincaré-Reeb tree is a complete binary tree and its vertices are endowed with a total order relation.
\end{theorem}

It is well-known that any algebraic curve has only finitely many non-generic directions of projection. After presenting a brief proof of this classical result (Theorem \ref{prop:FiniteNonGeneric}), we show that for asymptotic families of level curves near a strict local minimum of a real bivariate polynomial there are finitely many intervals on the real projective line outside of which all the projection directions are generic with respect to all the curves in the family:

\begin{theorem}[Theorem \ref{th:almostallgeneric}]\label{th:almostallgenericIntro}
There exists a real number $\nu>0$, and there are finitely many intervals $U_i\subset \mathbb{RP}^1$, of length $\nu$, such that for any $0<\varepsilon\ll \nu,$ all the projection directions of $\mathbb{RP}^1\setminus \bigcup_{i}U_i$ are generic (as in Definition \ref{DefGenericDirection}). 
\end{theorem}

Moreover, in Proposition \ref{prop:red} and Proposition \ref{prop:homeo}, we express the conditions for the vertical projection to be a generic projection for a given level curve, in terms of the polar curve and the discriminant locus of a certain polynomial map. This allows us to obtain the following reformulation of Theorem \ref{mainThIntro} from above:

\begin{theorem}[Theorem \ref{mainTh}]\label{mainTh:intro}
Let us consider the polar curve of $f$ with respect to the vertical projection $x$, that is $\Gamma(f,x):=\big\{(x,y)\in\mathbb{R}^2\mid \frac{\partial f}{ \partial y}(x,y)=0\big\}$, and the map $\phi:\mathbb{R}^2_{x,y}\rightarrow\mathbb{R}^2_{x,z}$, given by $\phi(x,y):=(x,f(x,y)).$

If the following two hypotheses are satisfied:

(a) the polar curve is reduced,

(b) he restriction of the map $\phi$ to the polar curve is a homeomorphism on the image $\phi(\Gamma)$,

\noindent then the generic asymptotic Poincaré-Reeb tree of $f$ relative to $x$ is a complete binary tree (i.e. every internal vertex has exactly two children), such that the preorder defined by $x$ on its vertices is a total order which is strictly monotone on the geodesics starting from the root. 
\end{theorem}

The choice of a generic projection will allow us to associate new combinatorial objects to the level curves: we will encode the shapes via a special class of permutations. In particular, as an application of Theorem \ref{mainTh:intro}, we show the following result:

\begin{theorem}[Theorem \ref{th:snakesBivar}]
Under the genericity hypotheses (a) and (b) from Theorem \ref{mainTh:intro} above, the local shape of level curves of a bivariate polynomial near a strict local minimum can be encoded in terms of 
alternating permutations, called snakes.
\end{theorem} 

See Theorem \ref{th:snakesBivar} for a more precise formulation. These special permutations turn out to be alternating permutations, that we call \emph{snakes} (see Figure \ref{fig:snakeRightIntro}). To this end, we study the local behaviour of the branches of the real polar curve in the neighbourhood of the origin, via the notions of \emph{crest} and \emph{valley}, inspired from \cite{CP}, where these were used to study atypical values at infinity. We give detailed proofs for all geometric phenomena we emphasize. 

{\centering
\includegraphics[scale=0.2]{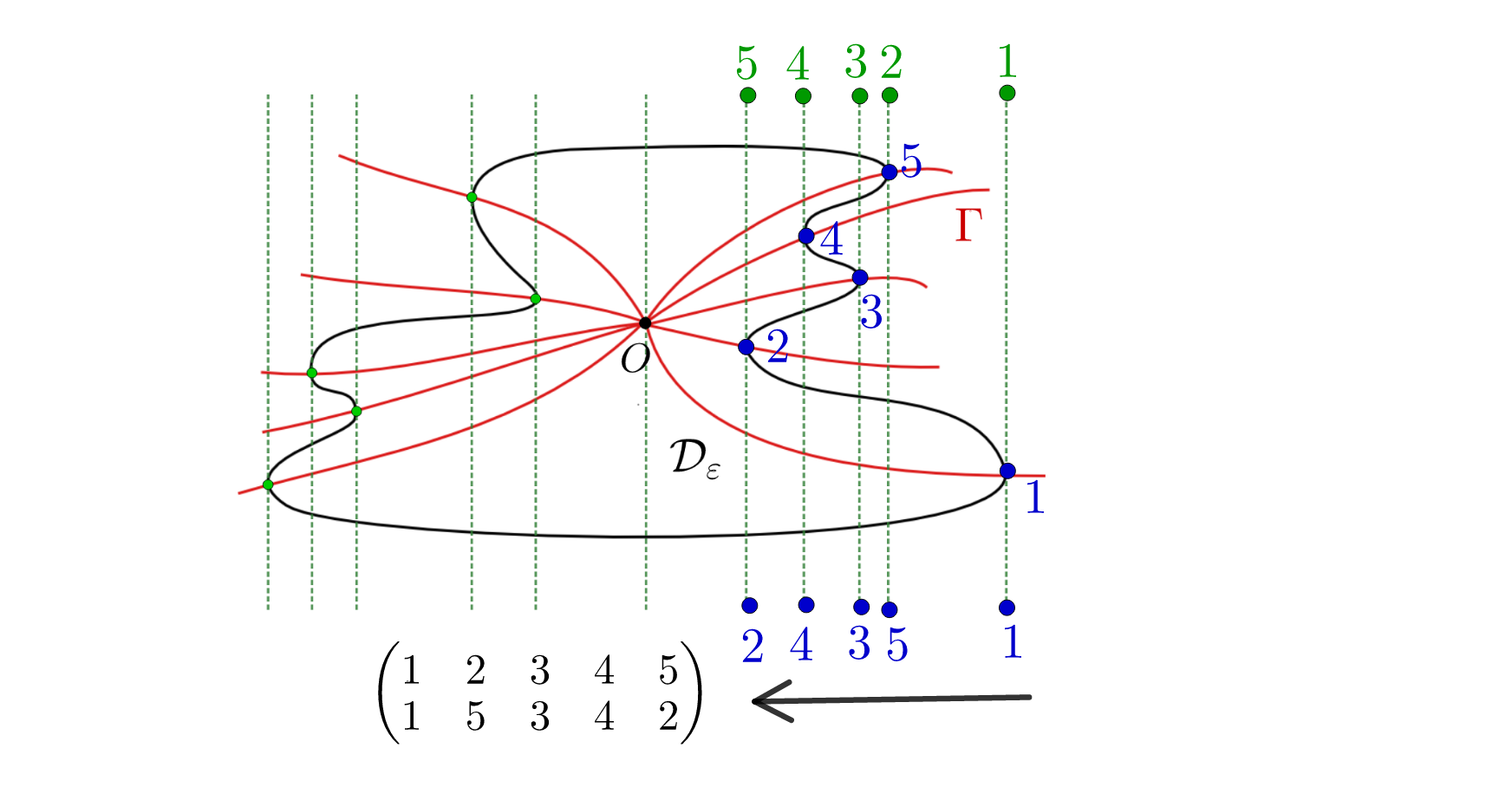} 
\captionof{figure}{A snake to the right (i.e. for $x>0$) associated to a smooth, compact, connected component of a real algebraic plane curve, near a strict local minimum.\label{fig:snakeRightIntro}}}
\vspace{10pt}

\subsection{Structure of the paper}
In Section \ref{SectionGenericDir} we define the notion of generic direction for an algebraic curve and we deduce that there are finitely many non-generic directions with respect to the curve. This is done via the classical study of the dual curve. Next, we focus our attention on the genericity of projections, in the case of asymptotic families of level curves of real bivariate polynomials near a strict local minimum. The end of Section \ref{SectionGenericDir} is dedicated to a geometric interpretation of the genericity hypotheses on the direction, in terms of the polar curve and the discriminant curve of a certain polynomial map. In Section \ref{sect:genAsyPRtrees} we present the generic properties of asymptotic Poincaré-Reeb trees. A new combinatorial interpretation of generic asymptotic Poincaré-Reeb trees, in terms of alternating permutations, called snakes, is given in Section \ref{sect:snakesCrestsValleys}. To this end, we carry a topological and combinatorial study of crests and valleys of real algebraic curves, illustrating how the branches of the real polar curve control the shape near a strict local minimum. 

\section*{Acknowledgements}
The results presented in this work constitute a part of my PhD thesis (see \cite{So1}), defended at \href{http://math.univ-lille1.fr/}{\textit{Laboratoire Paul Painlevé}}, \href{https://www.univ-lille.fr/}{\textit{Université de Lille, France}} and financially supported by \href{http://math.univ-lille1.fr/~cempi/}{\textit{Labex CEMPI}} (ANR-11-LABX-0007-01) and by \href{http://www.hautsdefrance.fr/}{\textit{Région Hauts-de-France}}. I am thankful to my PhD advisors, \href{http://math.univ-lille1.fr/~bodin/index.html}{Arnaud Bodin} and \href{http://math.univ-lille1.fr/~popescu/}{Patrick Popescu-Pampu}, who guided me during this work. I express my gratitude towards \href{http://ergarcia.webs.ull.es/}{Evelia R. Garc\' ia Barroso} and \href{https://webusers.imj-prg.fr/~ilia.itenberg/}{Ilia Itenberg} for reviewing my PhD thesis and towards \href{http://perso.ens-lyon.fr/ghys/accueil/}{\' Etienne Ghys} for being the president of the thesis committee.

\section{Generic directions of projection}\label{SectionGenericDir}
Generic projections are a subject of strong interest in real algebraic geometry (see \cite[Section 11.6]{RoyAlgRAG}) and in computational aspects of the topology of real algebraic varieties (see \cite{CLPPRT}). The problem of bitangents and inflections of algebraic curves is classical (see \cite{MR2856401} for Arthur Cayley's works). Recent progress in this subject has been made for instance by Brugallé, Itenberg, Sturmfels (see for example \cite{Bru1}, \cite{II2}, and \cite{Stu1}). 

\subsection{Choosing the direction of a generic projection for one level curve}\label{subSectionGenericDir}
The following definition is inspired from \cite[Definition 5.28]{So3}.

\begin{definition}\label{def:GoodNeighb}
Let $f:\bR^2\rightarrow \bR$ be a polynomial function such that $f(0,0)=0$. A small enough neighbourhood $V$ of $(0,0)$ such that $f$ has no other strict local minima in $V$, except from the origin and such that $V\cap (f=0)=\{(0,0)\}$, is called \defi{a good neighbourhood} of the origin.
\end{definition}

\begin{definition}\cite[Definition 1.1]{So3}\label{def:levelCEpsilon}
Let $f:\mathbb{R}^2\rightarrow\mathbb{R}$ be a bivariate polynomial function vanishing at the origin and having a strict local minimum at  $O$. Denote the  \defi{connected component that contains the origin} of the set $f^{-1}([0,\varepsilon])$ by $\mathcal{D}_\varepsilon$, for $\varepsilon>0$. Let $\mathcal{C}_\varepsilon$ denote $\mathcal{D}_\varepsilon\setminus \Int \mathcal{D}_\varepsilon.$ 

\end{definition}

\begin{definition}\label{DefGenericDirection}
Let us consider the set $\mathcal{C}_{\varepsilon}$ as in Definition \ref{def:levelCEpsilon}, for $\varepsilon>0$. 
A direction $d$ which satisfies the following conditions:
 
 \begin{enumerate}
\item  $d$ is not the direction of any bitangent of $\mathcal{C}_{\varepsilon}$;

\item $d$ is not the direction of any tangent which passes through any inflection point of $\mathcal{C}_\varepsilon$;
 \end{enumerate}
is called a \defi{generic direction with respect to} $\mathcal{C}_\varepsilon.$

\end{definition}

\begin{notation}
Throughout this paper, by \enquote{sufficiently small $\varepsilon$} or \enquote{small enough $\varepsilon$} we mean: \enquote{there exists $\varepsilon_{0}>0$ such that, for any $0<\varepsilon<\varepsilon_{0}$, one has \ldots}. We denote this by $0<\varepsilon\ll 1.$
\end{notation}

\begin{definition}\cite[page 74]{Fi}\label{DefDualCurve}
Let $\mathcal{C}\subset \mathbb{CP}^2$ be an algebraic curve. Then $$\mathcal{C}^*:=\left \{L\in \mathbb{CP}^2 \mid L \text{ is tangent to } \mathcal{C} \text{ at some point } p\in\mathcal{C} \right \}$$ is called \defi{the dual curve of} $\mathcal{C}.$\index{dual curve}
\end{definition}

\begin{remark}
Definition \ref{DefDualCurve} is standard (see for instance \cite[pages 73-74]{Fi}, \cite[page 252]{BK}). We denote the dual projective space by $(\mathbb{CP}^2)^*$. To a point $$y=[y_0:y_1:y_2]\in (\mathbb{CP}^2)^*,$$ it corresponds the line $V(y_0 X_0+y_1 X_1+y_2 X2)\subset \mathbb{CP}^2$.
 
The dual curve can be visually constructed by using the theory of pole-polar relation in a circle, see \cite[pages 577-583]{BK}.
\end{remark}

The three next results are classical. The proofs can be found, for instance, in \cite[Chapter 5]{Fi}, \cite[Chapter 7]{Wa} or \cite[Chapter 1]{GKZ}.

\begin{proposition}\cite[page 74]{Fi}\label{PropCurbaDuala}
Let $\mathcal{C}\subset \mathbb{CP}^2$ be an algebraic curve that has no lines as components. Then:

(a) the dual curve of $\mathcal{C}$, denoted by $\mathcal{C}^*$, is an algebraic curve;

(b) $\mathcal{C}^{**}=\mathcal{C};$

(c) if $\mathcal{C}$ is irreducible, then $\mathcal{C}^*$ is irreducible and $\deg \mathcal{C}^* \geq 2.$
\end{proposition}

\begin{lemma}\cite[page 20]{GKZ}\label{PropCorrespDual}
For a smooth curve $\mathcal{C}$, cusps of  $\mathcal{C}^*$ correspond to the simple inflection points of $\mathcal{C},$ i.e. such points where the tangent line has tangency to $\mathcal{C}$ of exactly the second order and is not tangent to $\mathcal{C}$ anywhere else. Similarly, nodes of $\mathcal{C}^*$ correspond to simple bitangents of $\mathcal{C}$, i.e. such that both tangencies are of first order and there are no other points of tangency.
\end{lemma}

\begin{lemma}\cite[page 49]{Fi}\label{PropNbSingul}
An irreducible algebraic curve $\mathcal{C}\subset \mathbb{CP}^2$ of degree n has at most $\frac{1}{2}(n-1)(n-2)$ singularities.
\end{lemma}

In the sequel, let us first prove a classical result, namely that any irreducible algebraic curve $\mathcal{C}\subset \mathbb{CP}^2$ has a finite number of bitangent lines and finitely many inflection points.

This will lead us later (see Subsection \ref{subs:asymptGeneric}) to a new result regarding the choice of a generic direction in an asymptotic setting, namely a direction which is generic with respect to $\mathcal{C}_\varepsilon$, for any small enough $\varepsilon>0$.

\begin{lemma}\label{PropFiniteBitg}
An irreducible algebraic curve $\mathcal{C}\subset \mathbb{CP}^2$ has a finite number of bitangent lines.
\end{lemma}

\begin{proof}
By Proposition \ref{PropCurbaDuala} (b), $\mathcal{C}^*$ is also an algebraic curve. Thus, by Lemma \ref{PropNbSingul}, the dual algebraic curve $\mathcal{C}^*$ has a finite number of double nodes, namely
$\frac{1}{2}(n^*-1)(n^*-2),$ where $n^*:=\deg \mathbb{C}^*$.
By Lemma \ref{PropCorrespDual}, the bitangents of $\mathcal{C}$ correspond to the double nodes of $\mathcal{C}^*$.
\end{proof}

A similar result for the inflection points can be given, but \cite{Fi} gave a better result, in function of $n:=\deg \mathcal{C}$, as follows:

\begin{lemma}\cite[page 68]{Fi}\label{PropFiniteInflec}
An algebraic curve $\mathcal{C}\subset \mathbb{CP}^2$ of degree $n:=\deg \mathcal{C} \geq 2$ that contains no lines has at most $3n(n-2)$ inflection points.
\end{lemma}

\begin{theorem}\label{prop:FiniteNonGeneric}
Let $\mathcal{C}_\varepsilon$ be as in Definition \ref{def:levelCEpsilon}. Then all but finitely many directions of projection are generic with respect to $\mathcal{C}_\varepsilon$, for a fixed $\varepsilon>0$.
\end{theorem}
\begin{proof}
The proof is a consequence of Lemma \ref{PropFiniteBitg} and Lemma \ref{PropFiniteInflec}, using Definition \ref{DefGenericDirection}.
\end{proof}

\subsection{Choosing the direction of a generic projection for an asymptotic family of level curves}\label{subs:asymptGeneric}
For sufficiently small $\varepsilon >0,$ the set $\mathcal{C}_\varepsilon$ is a smooth Jordan curve (\cite[Lemma 5.3]{So3}). In particular, it is a smooth compact component of the level curve $(f=\varepsilon)$. From now on, let us place ourselves in this asymptotic setting. The purpose of this section is to show that, except for finitely many small intervals on the circle of projection directions, all the directions are generic (in the sense of Definition \ref{DefGenericDirection}) for all $\mathcal{C}_\varepsilon$, where $\varepsilon>0$ is sufficiently small. More precisely, we prove the following result:

\begin{theorem}\label{th:almostallgeneric}
There exists a real number $\nu>0$, and there are finitely many intervals $U_i\subset \mathbb{RP}^1$, of length $\nu$, such that for any $0<\varepsilon\ll \nu,$ all the projection directions of $\mathbb{RP}^1\setminus \bigcup_{i}U_i$ are generic with respect to $\mathcal{C}_\varepsilon$ (as in Definition \ref{DefGenericDirection}). 
\end{theorem}

\begin{remark}
We consider the metric on $\mathbb{RP}^1$ to be the one  induced by the one on $\mathrm{S}^1/{\sim},$ where $\sim$ is the antipodal equivalence relation.
\end{remark}

\begin{proof}
$\bullet$ Firstly, let us consider the Hessian matrix of $f$, namely$$ \mathrm{Hess}\ (f)(x,y):=\begin{bmatrix}
\frac{\partial^2 f }{\partial x^2}(x,y) &\frac{\partial^2 f }{\partial x\partial y}(x,y) \\
\frac{\partial^2 f }{\partial y\partial x}(x,y) &\frac{\partial^2 f }{\partial y^2}(x,y) 
\end{bmatrix}.$$

Denote by $$\mathcal{H}:=\{(x,y)\in\mathbb{R}^2\mid \det\mathrm{Hess}\ (f)=0\}$$ the set of inflection points of the level sets of $f.$ Thus $\mathcal{H}\subset\mathbb{R}^2$ is a real algebraic set of dimension less or equal to 1 (see Figure \ref{fig:semialgTHB}, in green).

{\centering\vspace{10pt}
\includegraphics[scale=0.25]{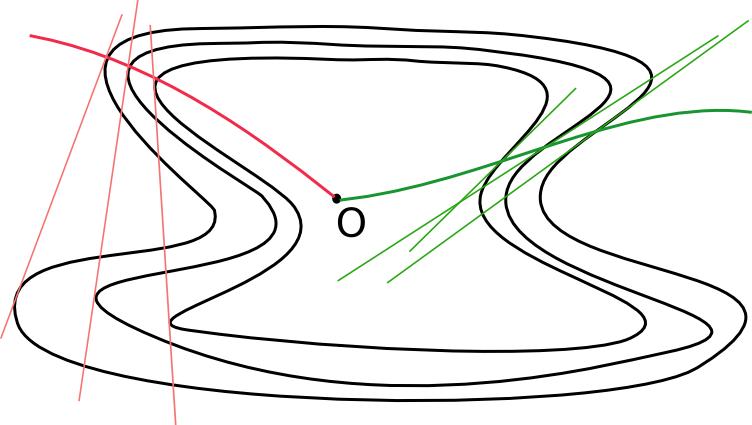} 
\captionof{figure}{In green, a half-branch $\gamma_{\mathcal{H}}\subset \mathcal{H}$. In red, a half-branch $\gamma_{\mathcal{B}}\subset\mathcal{B}$. The curves in black represent three curves $\mathcal{C}_\varepsilon$, for sufficiently small $\varepsilon>0$, in a good neighbourhood of the origin.\label{fig:semialgTHB}}}
\vspace{10pt}

$\bullet$ Secondly, let $$\mathcal{B}:=\{P\in\mathbb{R}^2\mid \exists Q\in\mathbb{R}^2, Q\neq P, \exists \varepsilon>0 \text{ such that } f(P)=f(Q)=\varepsilon \text{ and } P+T_P \mathcal{C}_\varepsilon=Q+T_Q \mathcal{C}_\varepsilon\}$$ be the set of points where there exists a bitangent to a level set of $f$. Here $T_P \mathcal{C}_\varepsilon$ is the tangent to $\mathcal{C}_\varepsilon$ at the point $P$. We want to prove that $\mathcal{B}$ is a semialgebraic set of dimension less or equal to 1 (see Figure \ref{fig:semialgTHB}, in red). 

Let $$\Sigma:=\{(P,Q)\in\mathbb{R}^2\times\mathbb{R}^2\mid  \exists \varepsilon>0 \text{ such that } f(P)=f(Q)=\varepsilon \text{ and } P+T_P \mathcal{C}_\varepsilon=Q+T_Q \mathcal{C}_\varepsilon\}$$ and let $$\delta:=\{(P,P)\in\mathbb{R}^2\times\mathbb{R}^2\mid \exists \varepsilon>0 \text{ such that } f(P)=\varepsilon \text{ and } P+T_P \mathcal{C}_\varepsilon=P+T_P \mathcal{C}_\varepsilon\}$$ be the diagonal of $\Sigma.$ Thus $\Sigma\setminus \delta$ is a semialgebraic set of $\mathbb{R}^2\times\mathbb{R}^2.$ Consider the projection $\Pi:\mathbb{R}^2\times\mathbb{R}^2\rightarrow\mathbb{R}^2,$ $\Pi(P,Q):=P.$ By Tarski-Seidenberg principle (see \cite[Section 2.1.2]{Co1}), $\Pi(\Sigma\setminus\delta)$ is a semialgebraic set of $\mathbb{R}^2$. Since $\Pi(\Sigma\setminus\delta)=\mathcal{B}$, $\mathcal{B}$ is a semialgebraic set. 

Let us argue by contradiction to prove that the dimension of the set $\mathcal{B}$ is at most $1$. Suppose $\mathcal{B}$ had dimension $2$. Take $P\in\mathcal{B}$. Then there exists a disk centred at $P$ included in $\mathcal{B}$. If $\varepsilon_P:=f(P)>0$, we obtain that there exists a level curve $\mathcal{C}_{\varepsilon_P}$, which intersects this disk in infinitely many points. This in impossible, by Theorem \ref{prop:FiniteNonGeneric}: for a fixed $\varepsilon_P>0$ there is a finite number of bitangents to $\mathcal{C}_\varepsilon$. In conclusion $\mathrm{dim}\ (\Pi(\Sigma\setminus \delta))\leq 1.$ 

$\bullet$ Therefore 
 $\mathring{\mathcal{H}}:=\mathcal{H}\setminus \{O\}$ and $\mathring{\mathcal{B}}:=\mathcal{H}\setminus \{O\}$ are semialgebraic sets of dimension less or equal to $1$. Hence the set $\mathring{\mathcal{H}} \cup \mathring{\mathcal{B}}$ is also a semialgebraic set of dimension at most $1$ in $\mathbb{R}^2\setminus\{O\}$. Thus, it has a finite number of half-branches at the origin (see \cite[Definition 5.13]{So3}).

$\bullet$ The next step of the proof is to define the polynomial map $\Phi:\mathbb{R}^2\rightarrow \mathbb{R}^2,$ $$\Phi(x,y):=\left (\frac{\partial f}{\partial x}(x,y),\frac{\partial f}{\partial y}(x,y)\right ).$$ The set $\mathring{\mathcal{H}} \cup \mathring{\mathcal{B}}$ is semialgebraic of dimension at most $1$, it is the union of finitely many half-branches: $\mathring{\mathcal{H}} \cup \mathring{\mathcal{B}}=\cup \gamma_i,$ for finitely many $i$. Namely we have $\gamma_i:\mathbb{R}\rightarrow\mathbb{R}^2.$

Let us fix $i_0$. The image by $\Phi$ of the half-branch $\gamma_{i_0}(t)$ is the arc $\Phi(\gamma_{i_0}(t))$. The slope of a secant of the half-branch $\Phi(\gamma_{i_0}(t))$ is given by $$\frac{\frac{\partial f}{\partial y}(\gamma_{i_0}(t))-0}{\frac{\partial f}{\partial x}(\gamma_{i_0}(t))-0}.$$ Since $\gamma_{i_0}$ is a semialgebraic half-branch, there exists the limit $$\lim_{t\rightarrow 0}\frac{\frac{\partial f}{\partial y}(\gamma_{i_0}(t))}{\frac{\partial f}{\partial x}(\gamma_{i_0}(t))}:=p_{i_0},$$ which is the slope of the tangent at the origin. In other words, we use the well-known fact that an algebraic half-branch never reaches the origin as an infinite spiral (see for instance \cite[105]{Gh1}). If $p_{i_0}=\infty,$ then make a change of coordinates. Therefore we can suppose that $p_{i_0}\in\mathbb{R}$ without loss of generality.

$\bullet$ The next step is to consider the map $\psi:\mathbb{R}^2\rightarrow\mathbb{RP}^1,$ defined by 
$$\psi(x,y):=\left [\frac{\partial f}{\partial y}(x,y):\frac{\partial f}{\partial x}(x,y)\right ].$$ 

Hence $$\lim_{t\rightarrow 0}\psi (\gamma_{i_0}(t))=\lim_{t\rightarrow 0}\left [\frac{\partial f}{\partial y}(\gamma_{i_0}(t)):\frac{\partial f}{\partial x}(\gamma_{i_0}(t))\right ]=\lim_{t\rightarrow 0}\left [\frac{\frac{\partial f}{\partial y}(\gamma_{i_0}(t))}{\frac{\partial f}{\partial x}\gamma_{i_0}(t))}:1\right ]=[p_{i_0}:1].$$

Since there are finitely many half-branches, we obtain finitely many points $[p_{i}:1].$ For each of them, let us define a small neighbourhood $U_{i}\subset\mathbb{RP}^1$, of $p_i$. In conclusion, there exists a real number $\nu>0$ such that the directions corresponding to $\mathbb{RP}^1\setminus \cup_i U_i$ are all generic directions (in the sense of  Definition \ref{DefGenericDirection}), with each $U_i$ having length $\nu$.

\end{proof}

\subsection{Genericity hypotheses via the polar curve}\label{subs:reducedHomeo}
The concept of polar curve (see \cite[page 589]{BK}) appeared already in the work of J.-V. Poncelet (\cite{Pon}) and M. Plücker (\cite{Pl}). There has been considerable research on the complex polar curves starting around the year 1970, when researchers such as García Barroso (\cite{GB1}, \cite{GB2}, \cite{GB3}),  L\^ e (\cite{Le1}), Teissier (\cite{Te1}, \cite{Te2}, \cite[page 682]{FT}), P\l oski, Gwo\' zdziewicz, Maugendre (\cite{Mau}), renewed the theory of polar curves via important investigations and
contributions, with numerous applications (\cite[page 769]{FT}). However, less is known about the real polar curves.

Polar varieties play an increasingly important role in computational real algebraic geometry. For instance, during the last two decenies researchers have been using them as geometric tools in proving correctness and in finding complexity estimates of algorithms, in the study of the topology of real affine varieties and in finding real solutions of polynomial equations. See for example results found by Bank, Safey El Din, Giusti, Heintz, Mbakop, Mork, Piene, Schost (\cite{MR2035216}, \cite{MR2585564}, \cite{MR3335572} and references therein).

In this subsection, our goal is to express the genericity hypotheses on the direction of projection (as in Definition \ref{DefGenericDirection}) in terms of the polar curve associated to the given polynomial $f$, and to this given direction.

\begin{definition}\cite[Section 4]{GB1}\label{DefPolCurve}
Let $f:\bR^2\rightarrow\bR$ be a polynomial function. 
The set 
$$\Gamma(f,x):=\left \{(x,y)\in\bR^2 \mid\frac{\partial f}{\partial y}(x,y)=0\right \}$$
is called \defi{the polar curve of} $f$ \defi{with respect to}\index{polar curve} $x.$
\end{definition}

\begin{remark}
The polar curve $\Gamma(f,x)$ consists of all the points $(x,y)\in\bR^2$ where the level curves of $f$ have a vertical tangent. For example, see Figure \ref{fig:snakeRightIntro}, where the polar curve appears in red. If $f$ has a strict local minimum, then the polar curve does not contain the projection direction (see \cite[Proposition 5.6]{So3}).
\end{remark}

\begin{proposition}
There are no vertical tangents to $\mathcal{C}_\varepsilon$, for $\varepsilon>0$ sufficiently small, that pass through the origin.
\end{proposition}

\begin{proof}
Let $M_0(x_0,y_0)\in\mathcal{C}_\varepsilon$ be a point such that $\mathcal{C}_\varepsilon$ has a vertical tangent at $M_0$ and such that this tangent passes through the origin. Hence the equation of the tangent is $x_0\frac{\partial f}{ \partial x}(x_0,y_0)+y_0\frac{\partial f}{ \partial y}(x_0,y_0)=0$. In other words, the vectors $\overrightarrow{OM_0}$ and $\mathrm{grad}_f(x_0,y_0)$ are perpendicular, since their scalar product is zero. Since the tangent is vertical, the vector $\mathrm{grad}_f(x_0,y_0)$ is horizontal, thus $\overrightarrow{OM_0}$ is a vertical vector. 

Since $M_0$ was arbitrarily chosen, we conclude that there exists a sequence of points $M_i$ that satisfy this property, namely that the $Oy$-axis is included in the polar curve\index{polar curve} $\Gamma(f,x).$ By Proposition \cite[Proposition 5.6]{So3}), we obtain a contradiction. 
\end{proof}

\begin{proposition}\label{prop:red}
If the polar curve of $f$ with respect to $x$, that is $\Gamma(f,x)$, is reduced, then there are no vertical inflection tangents to the curve $\mathcal{C}_\varepsilon$, for $\varepsilon>0$ sufficiently small.\index{polar curve}
\end{proposition}

\begin{proof}
Fix $x_0>0$. The vertical tangent at the point $(x_0,y_0,0)$ to $\mathcal{C}_\varepsilon$, in the $xOy$ plane, if it exists, is the projection of the horizontal tangent at the point $(x_0,y_0,\varepsilon)$ to the graph of the one variable polynomial $f(x_0,y)$, in the $(x=x_0)$ plane. See Figure \ref{fig:inflection3D}.

{\centering
\includegraphics[scale=0.2]{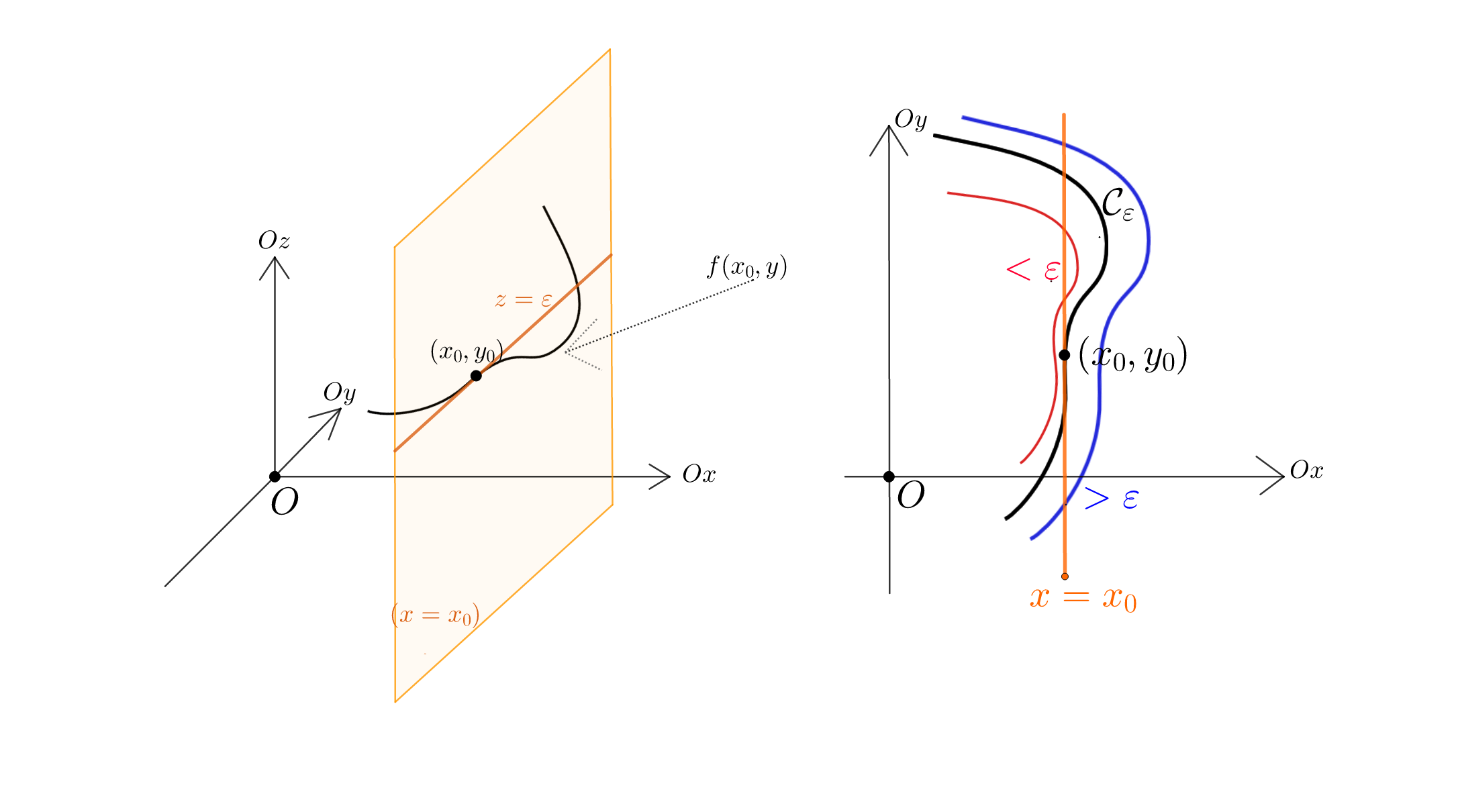} 
\captionof{figure}{An inflection point of the univariate polynomial $f(x_0,y),$ corresponding to a vertical inflection tangent of $\mathcal{C}_\varepsilon$.\label{fig:inflection3D}}}
\vspace{10pt}

For fixed $x_0>0$, we have $(x_0,y_0)\in\mathcal{C}_\varepsilon,$ that is $f(x_0,y_0)=\varepsilon.$ Suppose that $(x_0,y_0)$ is a vertical inflection point. Hence:

-if $y<y_0$, then $f(x_0,y)>\varepsilon;$

-if $y>y_0$, then $f(x_0,y)<\varepsilon$

(or its symmetric situation, that can be treated similarly).

This means that on the surface $\mathrm{graph}(f):=\{(x,y,z)\in\mathbb{R}^3\mid z=f(x,y)\}$ we have the corresponding two cases:

-if $y<y_0$, the graph of the one variable polynomial $f(x_0,y)$ is above $z=\varepsilon$;

-if $y>y_0$, the graph of the one variable polynomial $f(x_0,y)$ is below $z=\varepsilon$.

In other words, for a fixed sufficiently small $x_0>0$, the vertical inflection points of $\mathcal{C}_\varepsilon$ correspond to the inflection points of the one variable polynomial $f(x_0,y).$ This is due to the fact that in the univariate case, the polynomial $f(x_0,y)$ has an inflection point $(x_0,y_0)$ if and only if $(x_0,y_0)$ is a double root of the derivative of $f(x_0,y),$ that is a double root of the equation $\frac{\partial f}{\partial y}(x_0,y)=0.$

Since the above mentioned is true for any $x_0>0$ small enough, this means that there is a family of double roots of $\frac{\partial f}{\partial y}(x,y)=0.$ In other words, the equation of the polar curve has factors with multiplicity greater than one, thus the polar curve is not reduced.

Contradiction.
\end{proof}

\begin{proposition}\label{prop:homeo}
Let us consider the polar curve of $f$ with respect to $x$, that is $\Gamma(f,x),$ and the map $\phi:\mathbb{R}^2_{x,y}\rightarrow\mathbb{R}^2_{x,z},$ given by $\phi(x,y):=(x,f(x,y)).$ If the restriction of the map $\phi$ to the polar curve is a homeomorphism onto its image $\phi(\Gamma)$, then there are no vertical bitangents to the curve  $\mathcal{C}_\varepsilon$, for $\varepsilon>0$ sufficiently small.
\end{proposition}

\begin{remark}
If the restriction of $\phi$ to $\Gamma$ is a homeomorphism onto $\phi(\Gamma)$, then any two distinct polar branches of $\Gamma$ have different images by $\phi$.
\end{remark}

\begin{proof}
Suppose that there exist vertical bitangents to $\mathcal{C}_\varepsilon$, i.e. there exists $x_0\in\mathbb{R},$ $x_0>0$ sufficiently small and there exist $y_1\neq y_2$ such that $f(x_0,y_1)=f(x_0,y_2)$, with vertical tangent both at $(x_0,y_1)$ and at $(x_0,y_2)$. To be more precise, this means that there exist two different polar branches $\gamma_1\neq\gamma_2$, with $(x_0,y_1)\in\gamma_1$, $(x_0,y_2)\in\gamma_2$ such that $(x_0,y_1)\neq (x_0,y_2)$ (see Figure \ref{fig:homeo}), and $\phi(x_0,y_1)=\phi(x_0,y_2)$. Hence $\phi$ is not a homeomorphism when restricted to $\Gamma$ on its image. 
Contradiction.
\end{proof}

{\centering\vspace{10pt}
\includegraphics[scale=0.17]{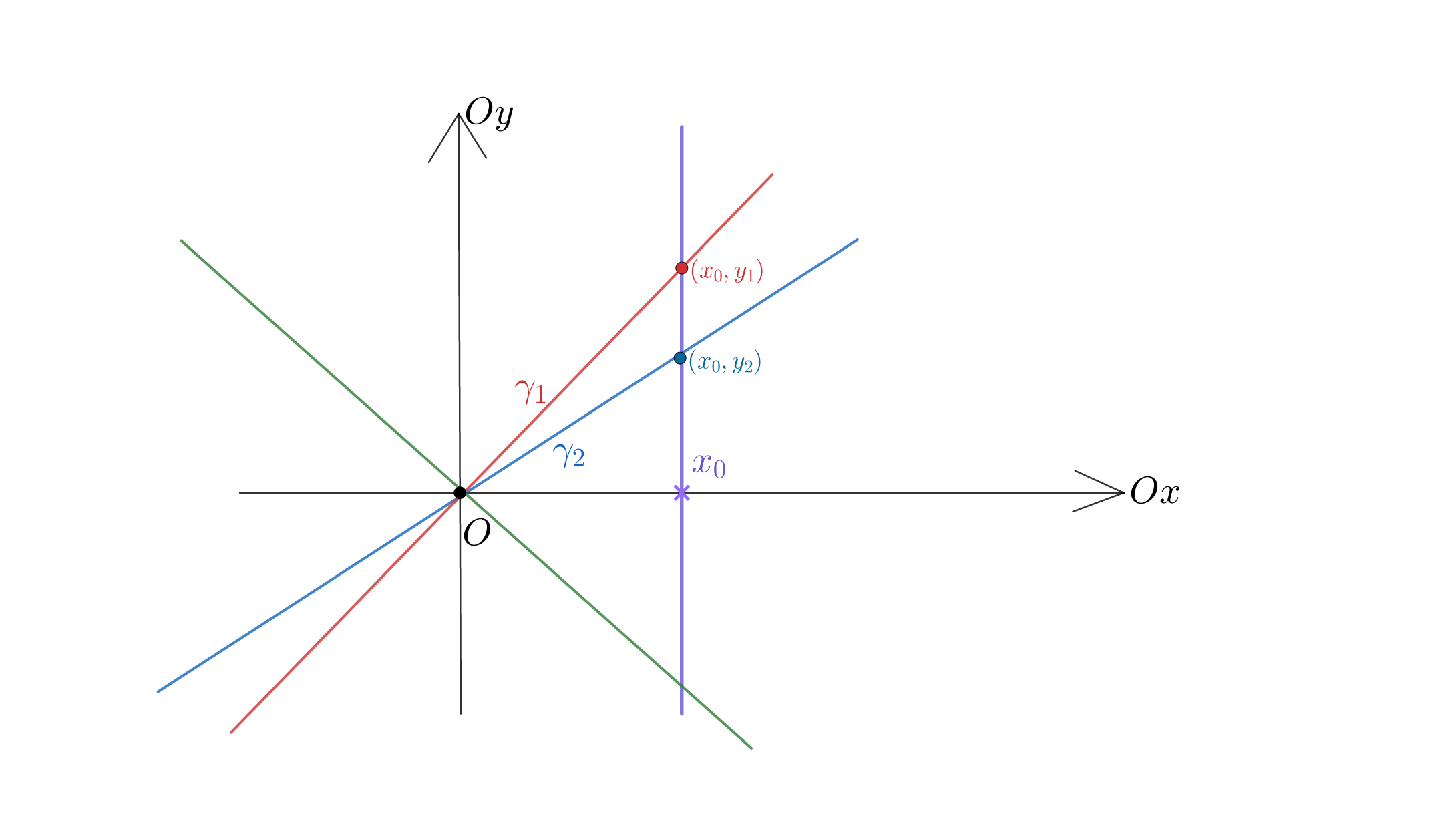} 
\captionof{figure}{Two different polar branches $\gamma_1\neq\gamma_2$, with $(x_0,y_1)\in\gamma_1$, $(x_0,y_2)\in\gamma_2$ such that $(x_0,y_1)\neq (x_0,y_2)$.\label{fig:homeo}}}
\vspace{10pt}

\section{Generic asymptotic Poincaré-Reeb trees}\label{sect:genAsyPRtrees}
Combinatorial objects are often used to encode the topological type of complex curve singularities: labeled trees (such as \emph{the Kuo-Lu tree}, \emph{the Eggers-Wall tree}, \emph{weighted dual graphs}; see \cite{MR3999063}, \cite{PPP19}), or polyhedral tools (such as \emph{the Newton polygon} used by Bodin in \cite{MR2129310}, \cite{MR2344205}). Another very useful simplicial complex called \emph{the lotus of a complex curve singularity} was first introduced by Popescu-Pampu in \cite{MR2920731} (see also \cite{Cas}).

In the real setting, combinatorial data of real curve singularities can be stored by plane trees such as the \emph{contact tree}. Such trees are completely described by Ghys in his recent book \cite{Gh1}, in terms of \emph{separable permutations} (see \cite{Gh2}), via the possible local configurations of any family $\{a_i(x)\}_i$ of real polynomials in one variable in a small enough neighbourhood of the origin of the real plane.

The graphs of these polynomials change their relative positions when they pass through a common zero (see Figure \ref{fig:ctcEx}), giving rise to permutations. Ghys proved that the permutations that can be obtained in this manner are exactly the so-called \enquote{separable} permutations.

{\centering\vspace{10pt}
\includegraphics[scale=1.2]{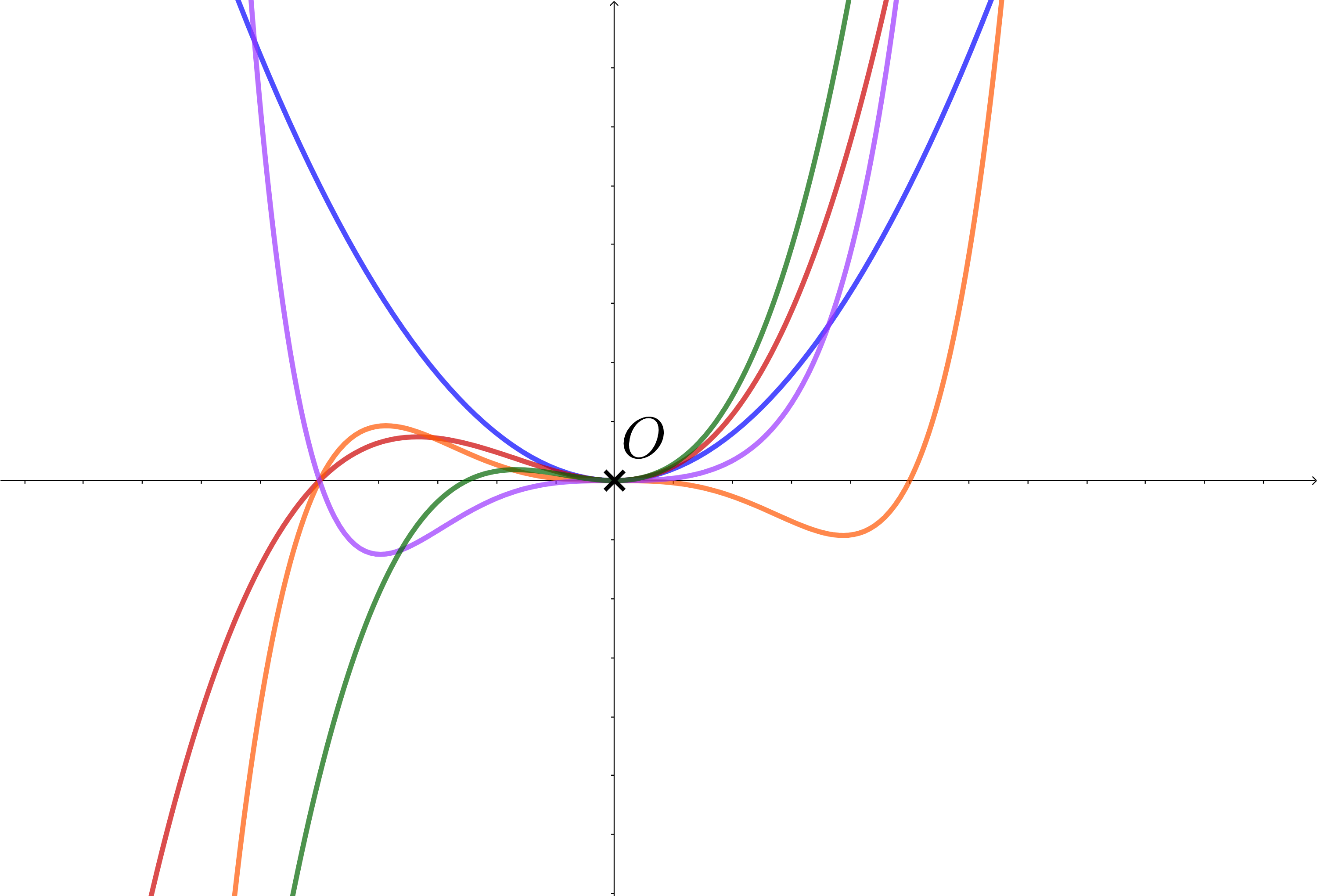} 
\captionof{figure}{The graphs of the polynomials $a_i(x)$ in a neighbourhood of the origin.\label{fig:ctcEx}}}
\vspace{10pt}

While the curves studied by Ghys were the zero locus of a bivariate real polynomial function, in this paper we consider a polynomial with the zero locus reduced to a point (at least in a small enough neighbourhood of the origin). We are interested in the shapes of the nearby level curves, which are called the real Milnor fibres of the polynomial at the origin.

Sorea introduced and described  the \emph{asymptotic Poincaré-Reeb trees} in \cite[Section 4]{So3} (see also \cite{So1}). Their purpose is to measure the non-convexity of real algebraic curves near strict local minima. 

Let us briefly recall the construction of the Poincaré-Reeb tree associated to a smooth, compact connected component of a real plane algebraic curve $\mathcal{C}_\varepsilon$ and to a chosen projection direction $\Pi:\mathbb{R}^2\rightarrow\mathbb{R},$ $\Pi(x,y):=x$. For a brief introduction in the standard vocabulary related to graphs and trees, we refer the reader to \cite[Subsection 2.1]{So2}.

 Denote by $\mathcal{D}_\varepsilon$ the topological disk bounded by $\mathcal{C}_\varepsilon$. Consider the following equivalence relation in $\mathbb{R}^2$: two points of $\mathcal{D}_\varepsilon$ are equivalent if they belong to the same connected component of a fibre of the projection $\Pi$. By taking the quotient map, we construct an object called the Poincaré-Reeb graph, which is a special type of plane tree (see Figure \ref{fig:1Coste}). In particular, its vertices are endowed with a total preorder induced by $x$ (see \cite[Corollary 4.21]{So3}). In the asymptotic case, that is, for small enough level curves near a strict local minimum at the origin, on each geodesic starting from the root (i.e. the image of the origin), this preorder is strictly monotone (\cite[Theorem 5.31]{So3}). This is what we call the asymptotic Poincaré-Reeb tree.

{\centering\vspace{10pt}
\includegraphics[scale=0.1]{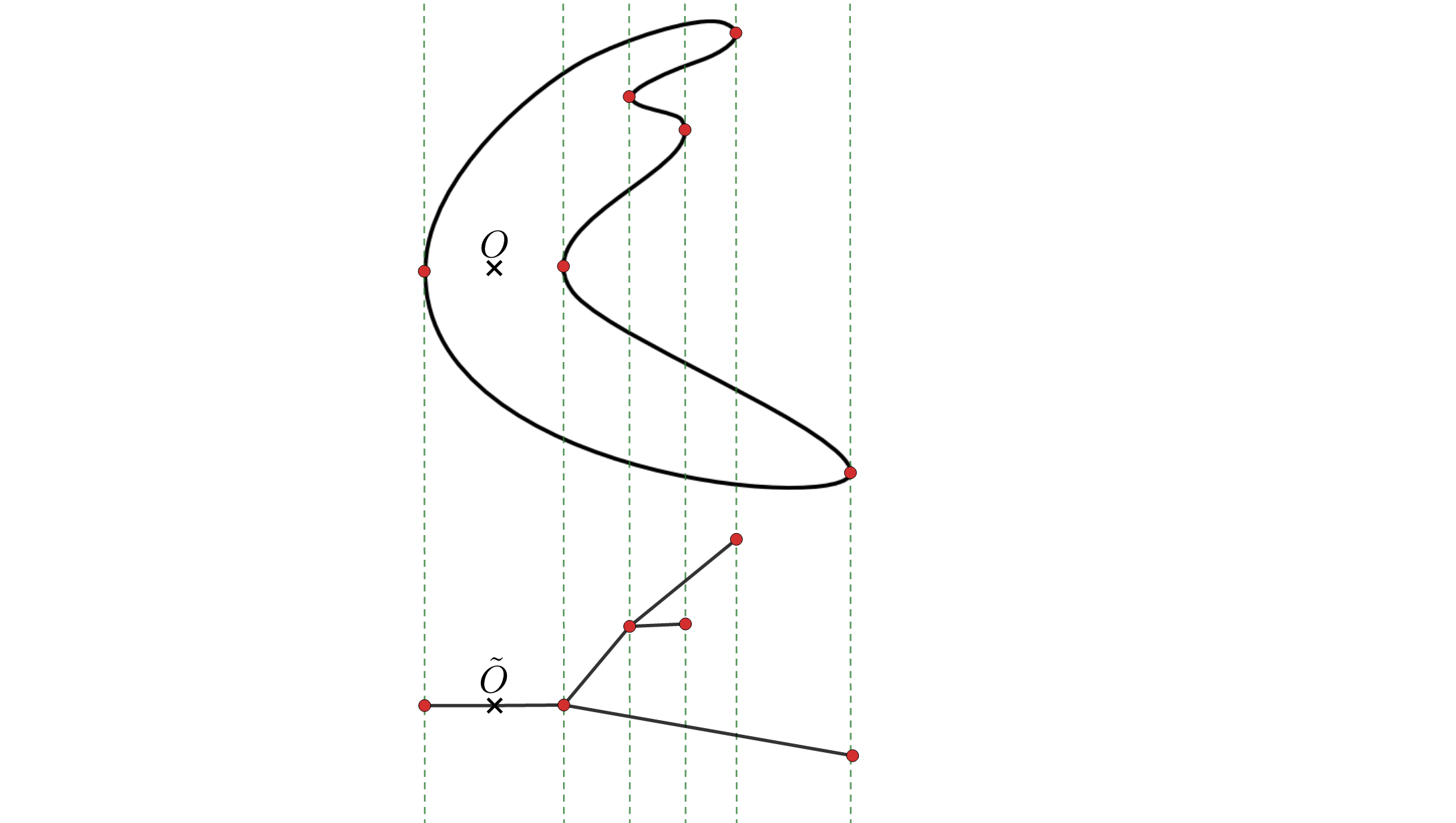}
\captionof{figure}{The Poincaré-Reeb tree associated to a smooth, compact connected component of a level curve near a strict minimum at the origin, with respect to the vertical projection.\label{fig:1Coste}}}
\vspace{10pt}

\subsection{The properties of a generic asymptotic Poincaré-Reeb tree}\label{subs:genAsPRT}
The aim of this section is to study the properties of asymptotic Poincaré-Reeb trees under the assumption that the direction of projection is generic with respect to $\mathcal{C}_\varepsilon$ (in the sense of Definition \ref{DefGenericDirection}).

\begin{definition}\label{def:genericAsymptPReeb}
If the direction of projection $x$ is generic, then the asymptotic Poincaré-Reeb tree is called \defi{a generic asymptotic Poincaré-Reeb tree} of $\mathcal{C}_\varepsilon$, with respect to the chosen direction.
\end{definition}

\begin{proposition}\label{cor:redus}
If the polar curve is reduced, then the asymptotic Poincaré-Reeb tree has no vertices of valency equal to two, except for the root.
\end{proposition}

\begin{proof}
It follows from Proposition \ref{prop:red} and from the construction of the  Poincaré-Reeb tree (\cite[Section 4]{So3}).
\end{proof}

\begin{proposition}\label{cor:homeo}
Let us consider the polar curve $\Gamma(f,x),$ and the map $\phi:\mathbb{R}^2_{x,y}\rightarrow\mathbb{R}^2_{x,z}$, given by $\phi(x,y):=(x,f(x,y)).$ If the restriction of the map $\phi$ to the polar curve is a homeomorphism on the image $\phi(\Gamma)$, then the asymptotic Poincaré-Reeb tree has no vertices of valency strictly bigger than $3$.
\end{proposition}

\begin{proof}
It follows from Proposition \ref{prop:homeo} and from the construction of the Poincaré-Reeb tree (\cite[Section 4]{So3}).
\end{proof}

\begin{example}\label{ex:reebExample}
In Figure \ref{fig:reeb} we show an example of a generic asymptotic Poincaré-Reeb tree.

{\centering\vspace{10pt}
\includegraphics[scale=0.12]{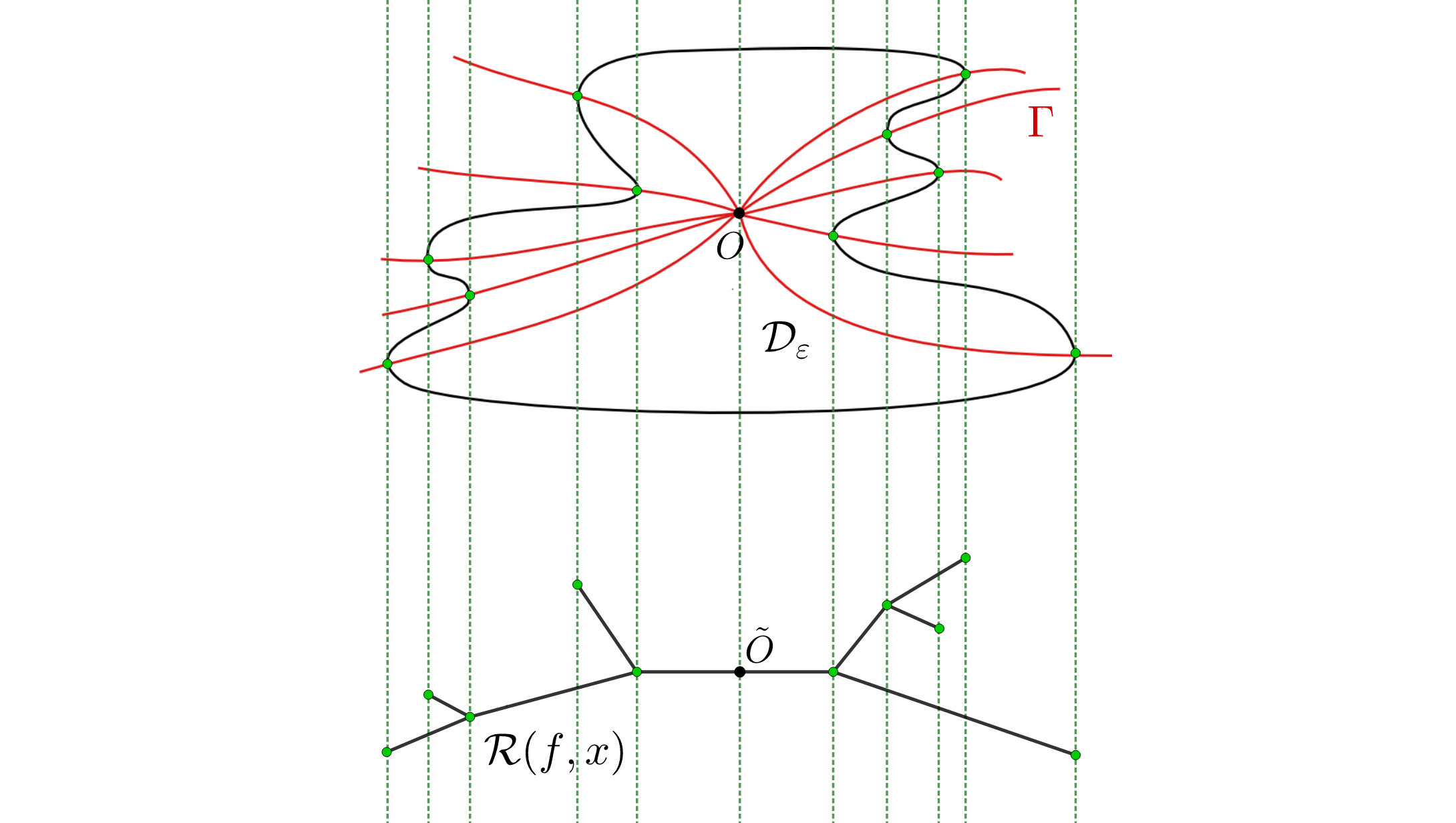} 
\captionof{figure}{A generic asymptotic Poincaré-Reeb tree.\label{fig:reeb}}}
\end{example}
\vspace{10pt}

\begin{example}
An example of two topologically inequivalent generic asymptotic Poincaré-Reeb trees is shown in Figure \ref{fig:preorderOrder1} below.

{\centering\vspace{10pt}
\includegraphics[scale=0.1, , trim={0cm 6cm 0 0cm}, clip]{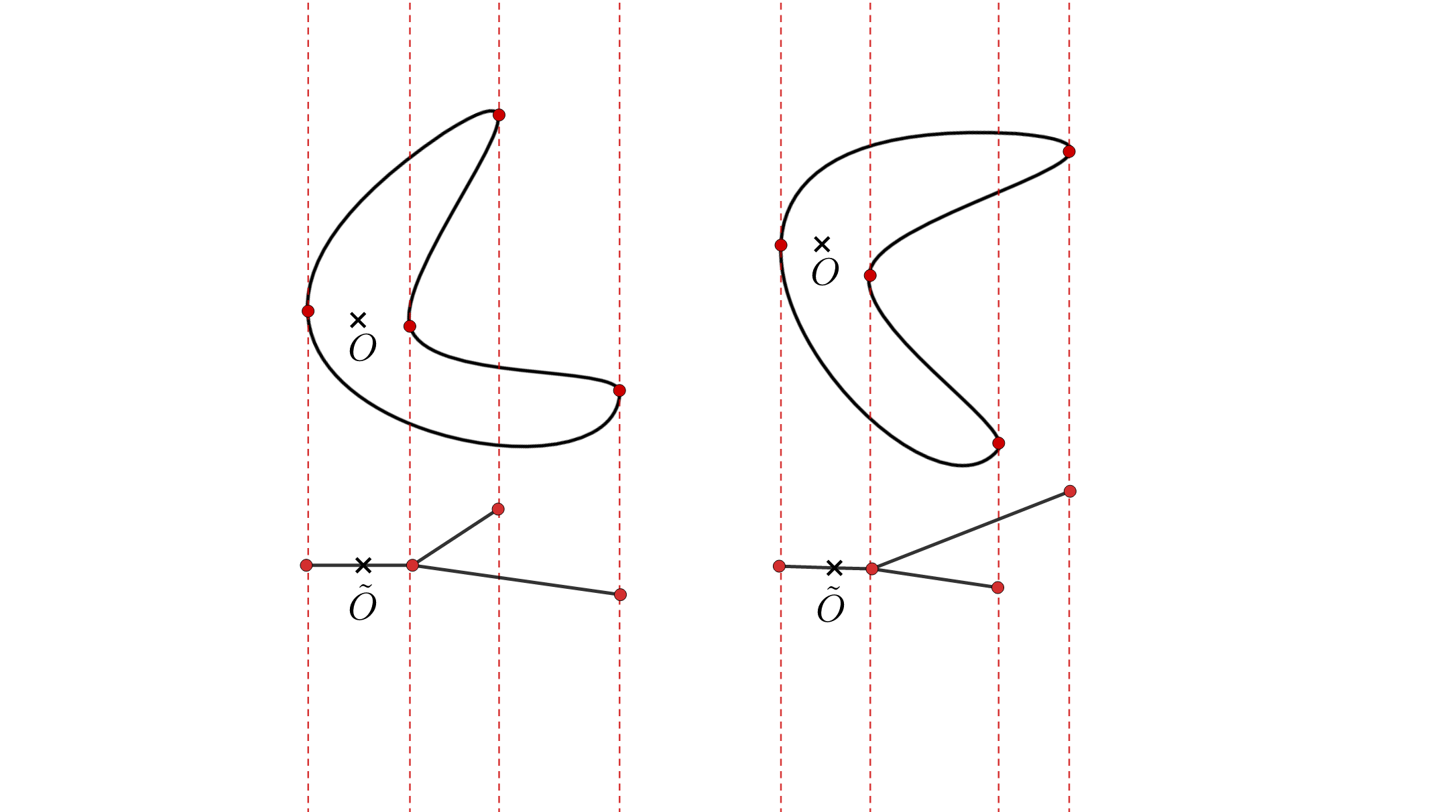} 
\captionof{figure}{Two topologically inequivalent generic asymptotic Poincaré-Reeb trees.\label{fig:preorderOrder1}}}
\vspace{10pt}

The lack of equivalence is due to the fact that in order to pass from one tree to the other one, we reach a non-generic tree, where the two leaves have the same $x-$coordinate, see \cite[Figure 4]{So3}.

\end{example}

\begin{theorem}\label{mainTh}
Let us consider the polar curve of $f$ with respect to $x$, that is $\Gamma(f,x)$, and the map $\phi:\mathbb{R}^2_{x,y}\rightarrow\mathbb{R}^2_{x,z}$, given by $\phi(x,y):=(x,f(x,y)).$

If the following two hypotheses are satisfied:

(a) the polar curve is reduced,

(b) the restriction of the map $\phi$ to the polar curve is a homeomorphism on the image $\phi(\Gamma)$,

\noindent then the generic asymptotic Poincaré-Reeb tree of $f$ relative to $x$ is a complete binary tree (i.e. every internal vertex has exactly two children) such that the preorder defined by $x$ on its vertices is a total order which is strictly monotone on the geodesics starting from the root. 
\end{theorem}

\begin{proof}
By Proposition \ref{cor:redus} and Proposition \ref{cor:homeo}, all its internal vertices have valency exactly $3$, thus the tree is complete binary. The total preorder endowing the vertices (see \cite[Corollary 4.21]{So3}) becomes a total order, since there are no vertical tangencies.
\end{proof}

\begin{figure}[H]
\centering
\begin{subfigure}[b]{0.45\textwidth} 
\includegraphics[scale=0.18, trim={6cm 0cm 0 4cm}, clip]{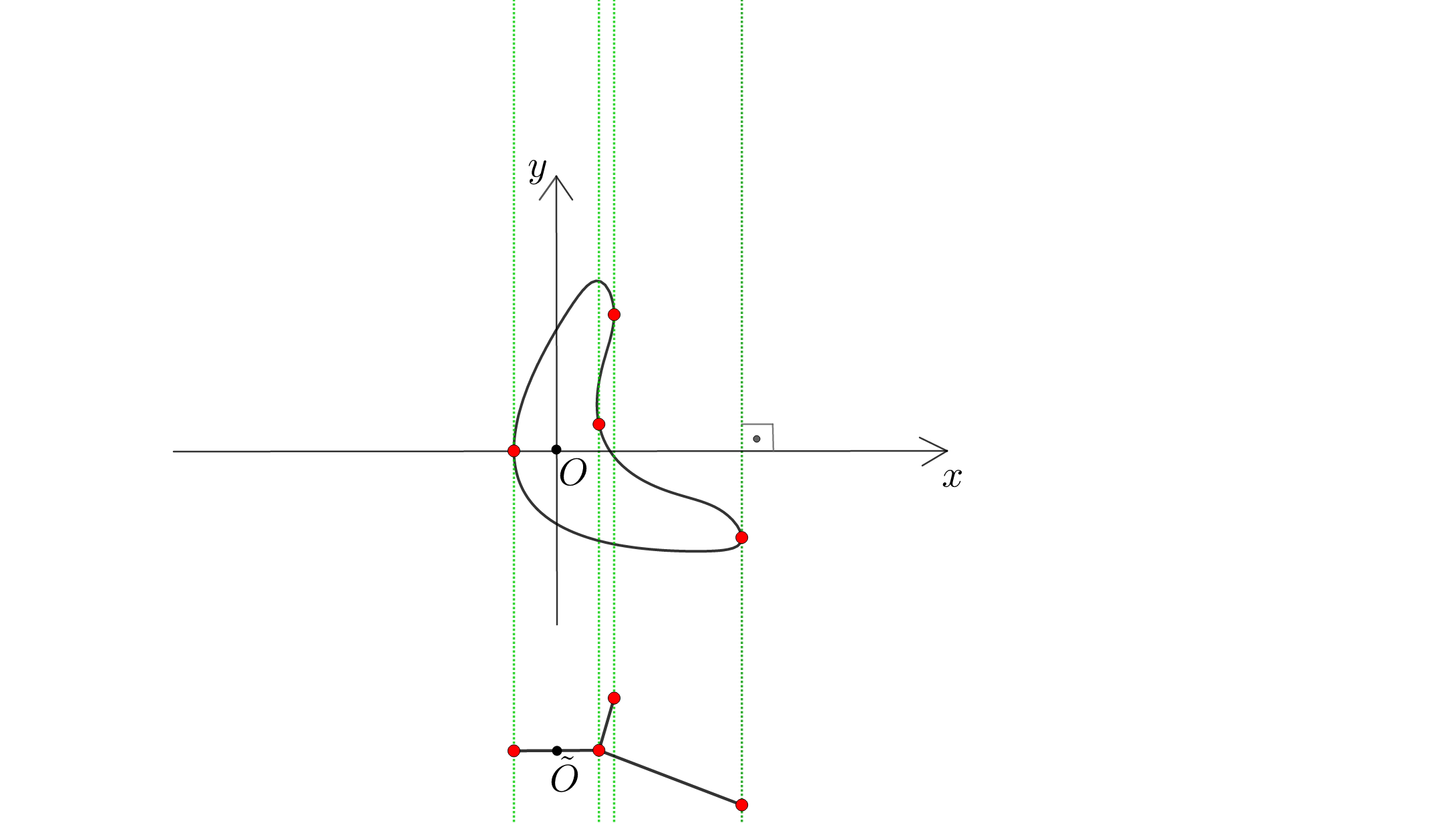} 
\end{subfigure}
\begin{subfigure}[b]{0.45\textwidth}
\includegraphics[scale=0.18, trim={6cm 1cm 0 3cm}, clip]{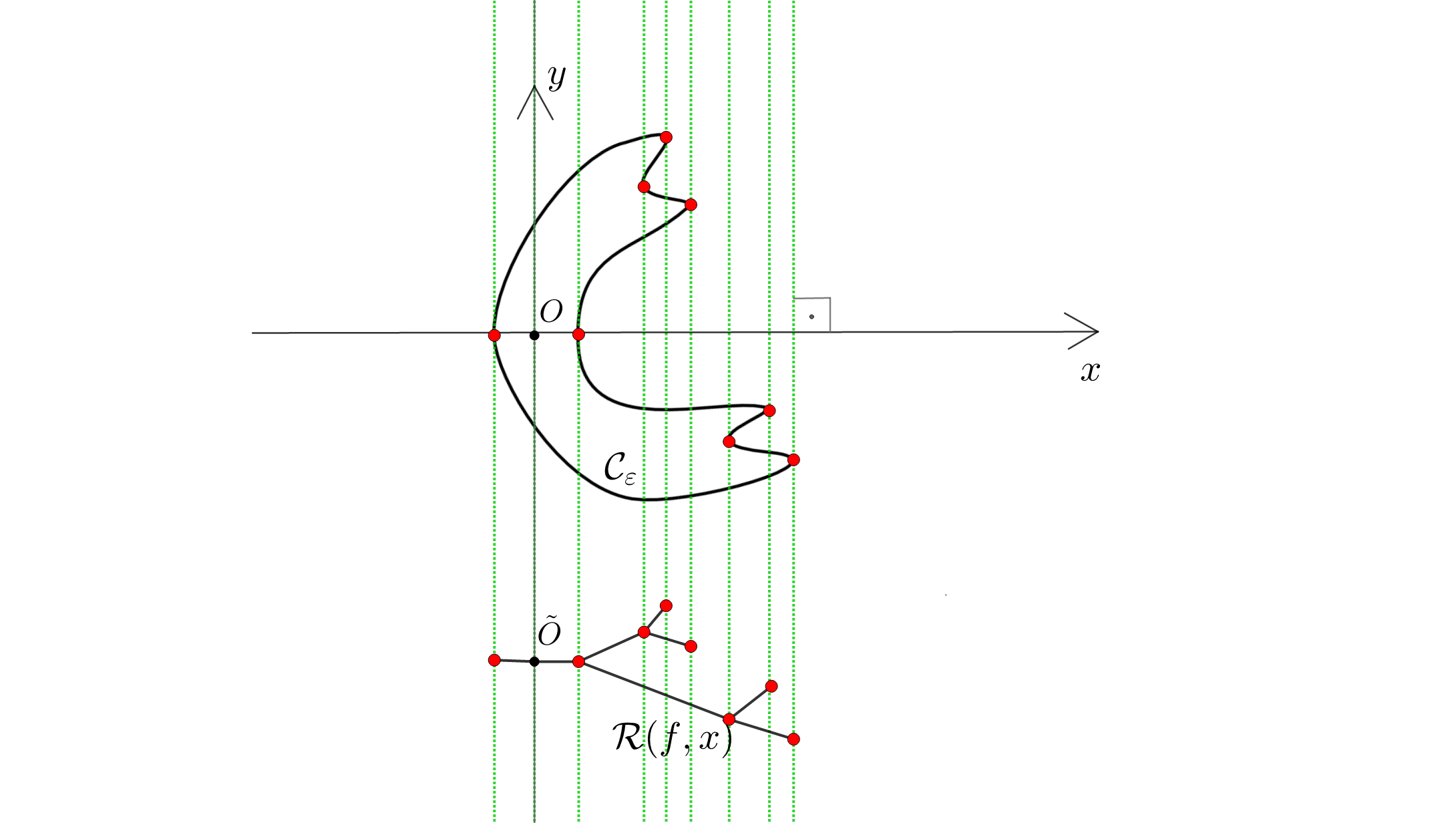}
\end{subfigure}

\caption{Two generic asymtotic Poincaré-Reeb trees.\label{fig:doubleBanReeb}}
\end{figure}

\subsection{Examples of generic asymptotic Poincaré-Reeb trees}
This subsection is dedicated to some examples of generic asymptotic Poincaré-Reeb trees.

\begin{example}
In Figure \ref{fig:doubleBanReeb}, two examples of generic asymptotic Poincaré-Reeb trees (associated to the vertical projection direction $x$) are shown. We denote them by $\mathcal{R}(f,x).$

\end{example}

\subsection{Examples of non-generic asymptotic Poincaré-Reeb trees}
If one considers polar curves $\Gamma(f,x)$ that are not reduced, or maps $\phi$ such that the restriction of $\phi$ on $\Gamma(f,x)$ is not a homeomorphism on its image, then one obtains non-generic Poincaré-Reeb trees. They have vertices whose valencies are not $3$, due to either bitangents or inflectional tangencies.

\begin{example}
\textbf{Bitangents:}

If the restriction of the map $\phi$ on $\Gamma(f,x)$ is not a homeomorphism on its image, then we obtain non-generic Poincaré-Reeb trees, since we do not have a total order on the vertices of the asymptotic Poincaré-Reeb tree. Vertices whose valencies are not $3$ appear.

Let us consider the polynomial $f:\mathbb{R}^2\rightarrow \mathbb{R},$ $$f(x,y):=x^{10}+\frac{y^6}{6}-3x\frac{y^4}{4}+x^2y^2.$$ 

The Poincaré-Reeb tree  of $f$ contains a vertex with valency equal to $4$. For a small enough $\varepsilon>0$, the curve $\mathcal{C}_\varepsilon$ is represented in blue in Figure \ref{fig:NotGenericEx}. The polar curve (in red) is $$\Gamma(f,x)=\{(x,y)\in\mathbb{R}^2\mid y(x-y^2)(2x-y^2)=0\}.$$

{\centering\vspace{10pt}
\includegraphics[scale=0.12]{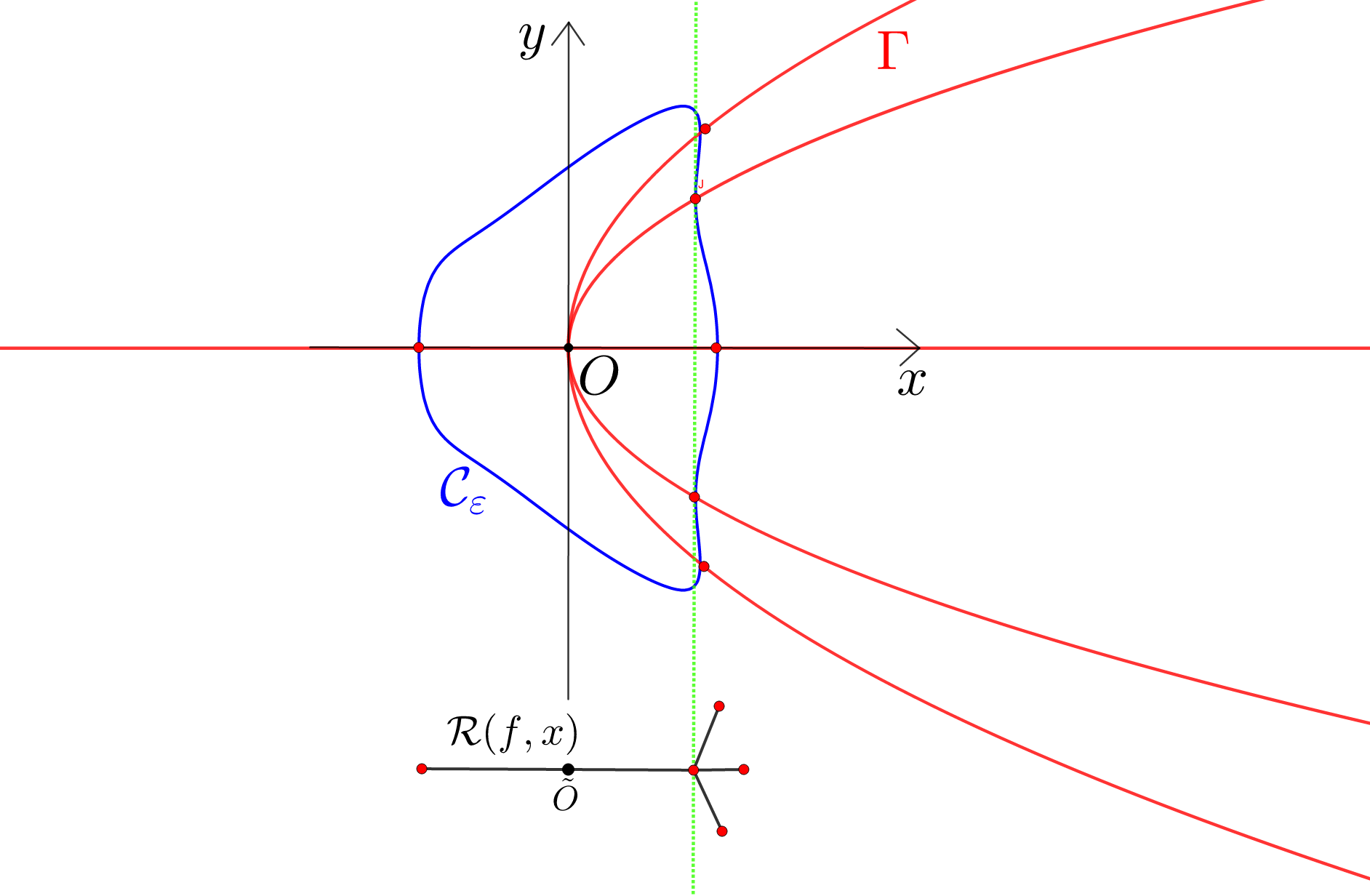} 
\captionof{figure}{A projection which is not generic giving vertices of valency greater than $3$ in the Poincaré-Reeb tree. \label{fig:NotGenericEx}}}
\vspace{10pt}
\end{example}

\begin{example}
\textbf{Inflections:}

If one considers the polar curves $\Gamma(f,x)$ not reduced, then one obtains non-generic Poincaré-Reeb trees, since the levels have vertical inflection tangents.

Consider the polynomial $f:\mathbb{R}^2\rightarrow \mathbb{R},$ $$f(x,y):=x^{12}+\int_{0}^y t^2(t-x^2)\mathrm{d}t.$$ Namely, the polar curve (in red) is: $$\Gamma(f,x)=\{(x,y)\in\mathbb{R}^2\mid y^2(y-x^2)=0\},$$ as one can see in Figure \ref{fig:nonGenericReebPolar} below.

{\centering\vspace{10pt}
\includegraphics[scale=0.13, trim={0 8cm 0 0cm}, clip]{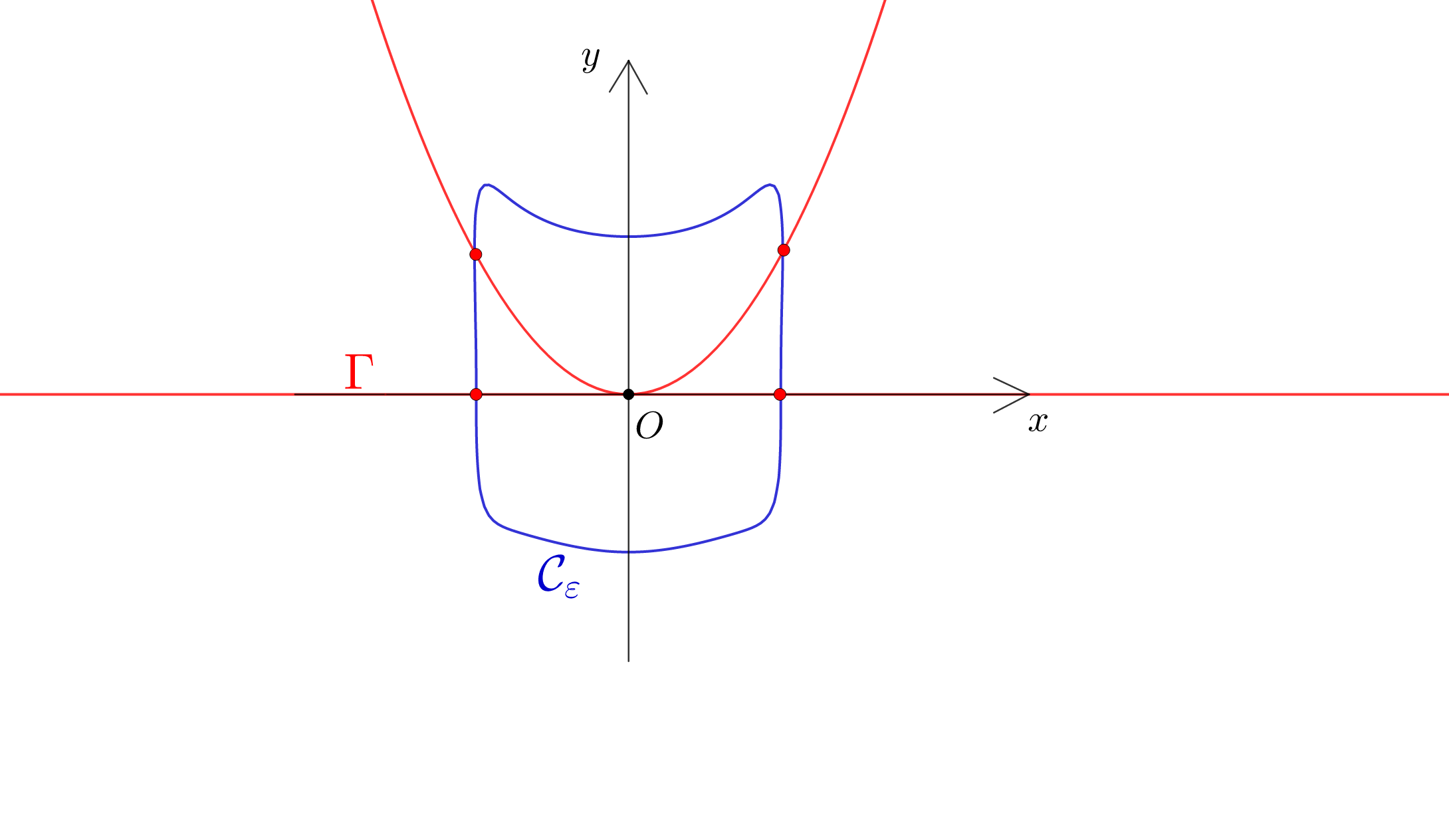} 
\captionof{figure}{The level curve $\mathcal{C}_\varepsilon$ of $f(x,y):=x^{12}+\int_{0}^y t^2(t-x^2)\mathrm{d}t$ and its non reduced polar curve $\Gamma(f,x)=\{(x,y)\in\mathbb{R}^2\mid y^2(y-x^2)=0\}$, in red.\label{fig:nonGenericReebPolar}}}
\vspace{10pt}

There are two vertical inflection points on the $Ox$ axis, that will give two vertices of valency $2$ in the Poincaré-Reeb tree, thus a non-generic asymptotic Poincaré-Reeb tree, shown in Figure \ref{fig:nonGenericReebGraph}. 

\begin{figure}[H]
\centering\vspace{10pt}
\includegraphics[scale=0.13, trim={6cm 8cm 2cm 7cm}, clip]{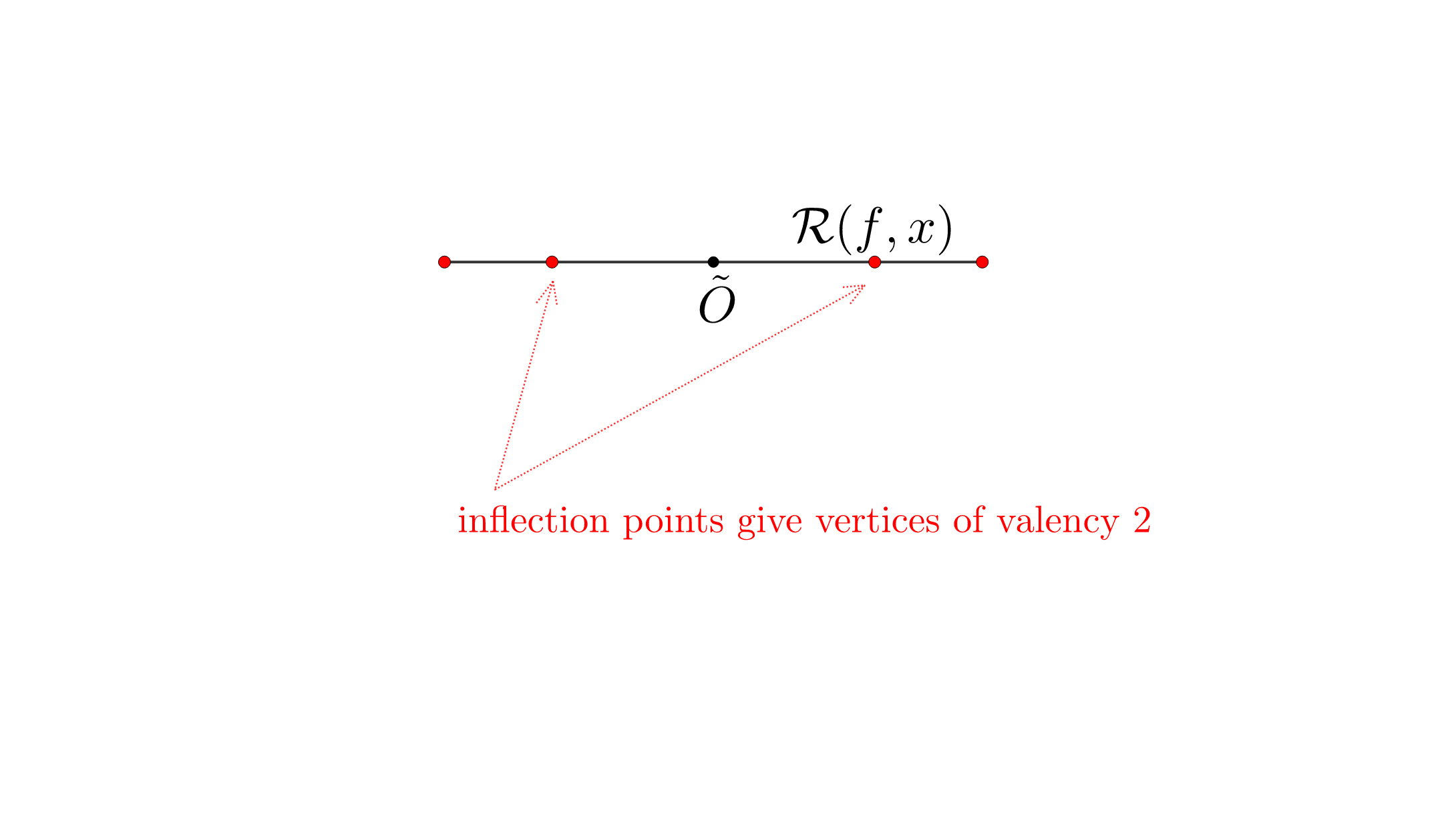} 
\caption{Non-generic asymptotic Poincaré-Reeb tree $\mathcal{R}(f,x)$ with two vertices of valency $2$, which correspond to two vertical inflection points.\label{fig:nonGenericReebGraph}}
\end{figure}
\end{example}

\begin{comment}
 \begin{remark}\label{Remark:FoliationNotLinear}
The Poincaré-Reeb tree is the image of the disk $\mathcal{D}_\varepsilon$ shrinked by considering the quotient topology as explained in Definition \ref{DefReeb}.  After the $\Pi_\mathcal{D}$ projection, the real plane is given the topological structure of a plane foliated by lines. A linear structure has the role of comparing lengths of segments. 

In the context of Definition \ref{def:transversalTree}, we shall forget the notion of length, because we want each line of the initial foliation to become only an one-dimensional manifold homeomorphic to a line. In other words, we keep in mind only the total order of the values on the bifurcation vertices like Arnold did. We do not keep in mind the exact values of the function on the bifurcation vertices.
 \end{remark}
\end{comment}

\section{Alternating permutations associated to level curves near a strict local minimum}\label{sect:snakesCrestsValleys}
The main goal of this section is to show that if the direction of projection is generic, then we can use alternating permutations to  encode the topology of the local shape of level curves of a bivariate polynomial near a strict local minimum. We call these permutations,  snakes.

\subsection{Positive tree, negative tree, union tree}\label{sect:RightLeftTree}
Note that by \cite[Definition 5.30, Theorem 5.31]{So3}, the image of the origin is the root of the asymptotic  Poincaré-Reeb tree. The real line is oriented, thus once the root is fixed, the root will separate without ambiguity between the negative side (i.e. left side) and the positive side (i.e. right side), as we shall see in the sequel.  

\begin{definition}
If $v_1(x_1,y_1)$ and $v_2(x_2,y_2)$ are two points in a real plane $P$ endowed with a trivial fibre bundle $\Pi:P\rightarrow\mathbb{R}$, $\Pi(x,y):=x$, we say that $v_2$ is \defi{to the right of} $v_1$ if $x_2>x_1.$ If $x_2<x_1,$ we say that $v_2$ is \defi{to the left of} $v_1$.
\end{definition}

\begin{definition}\label{def:proper binary right tree}
We call \defi{a positive tree}\index{tree!positive tree} a plane tree in which every vertex has its children to the right. We call \defi{positive edges}\index{edge!positive edge} the edges belonging to a complete plane binary positive tree.
\end{definition}

\begin{example}
The following Figure \ref{fig:rightTree} shows a complete plane binary positive tree. For instance, the vertex $v_2$ is to the right of the vertex $v_1$.

\begin{definition}\label{def:proper binary left tree}
We call \defi{a negative tree}\index{tree!negative tree} a plane tree in which every vertex has its children to the left. We call \defi{negative edges}\index{edge!negative edge} the edges belonging to a complete plane binary negative tree.
\end{definition}

{\centering\vspace{10pt}
\includegraphics[scale=0.15, trim={0 6cm 0 4cm}, clip]{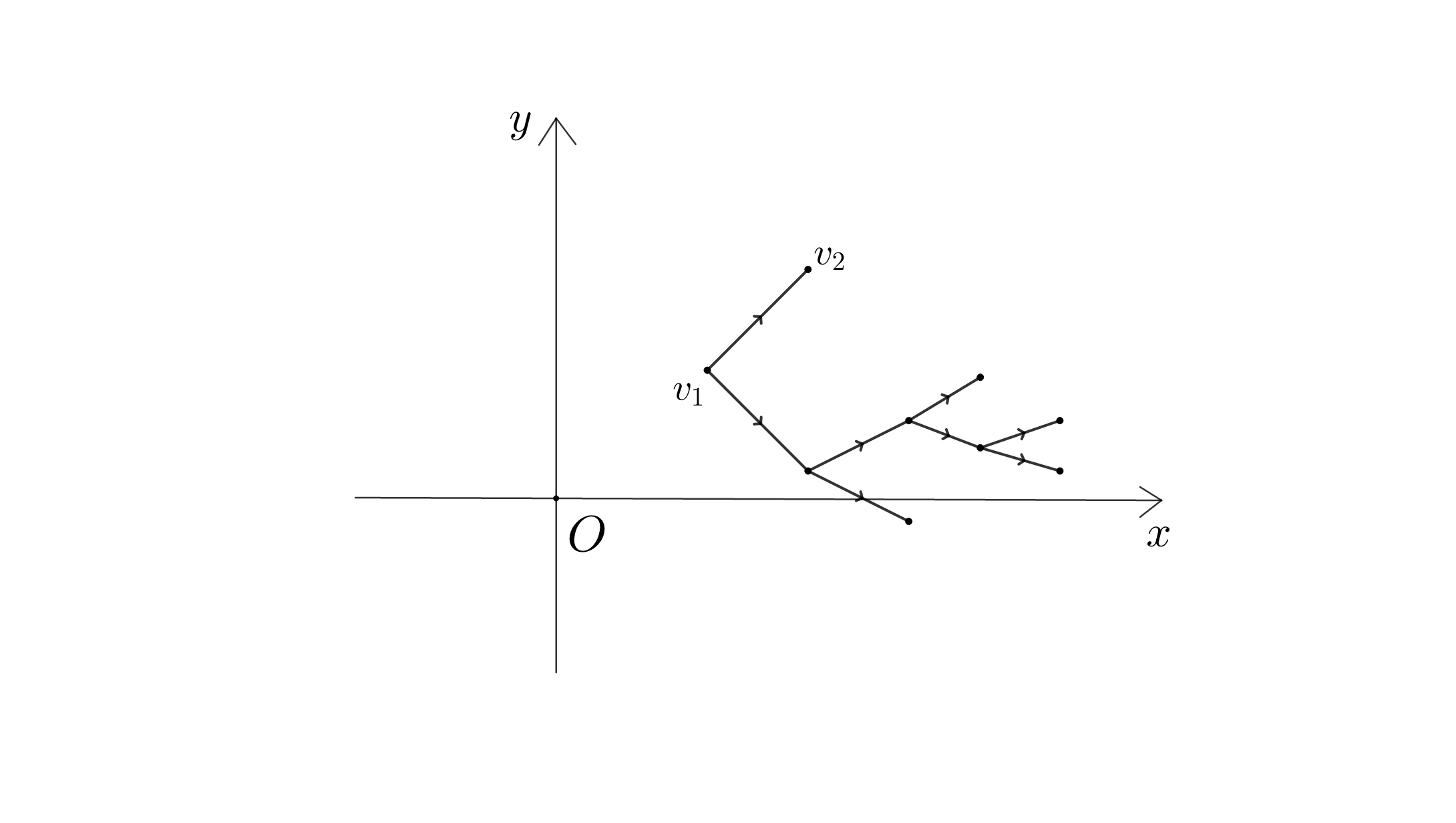} 
\captionof{figure}{A complete plane binary positive tree.\label{fig:rightTree}}}
\vspace{10pt}
\end{example}

\begin{definition}\label{def:unionTree}
Let $\mathcal{T}_+$ be a positive tree and $\mathcal{T}_-$ be a negative tree. We call \defi{the union tree}\index{tree!union tree} between $\mathcal{T}_-$ and $\mathcal{T}_+$, denoted by $\mathcal{T}_- - O - \mathcal{T}_+$, the plane tree obtained by making the connected sum of the two trees: we glue their roots into a point $O$, which becomes the root of the union-tree.
\end{definition}

\begin{example}
We show in Figure \ref{fig:unionTree} how we obtain the union tree, by taking the connected sum of two half-planes. The arrows show the orientation of the half-planes, such that at the glued line the orientations are reversed.

{\centering\vspace{10pt}
\includegraphics[scale=0.2, trim={0cm 0cm 0 4cm}, clip]{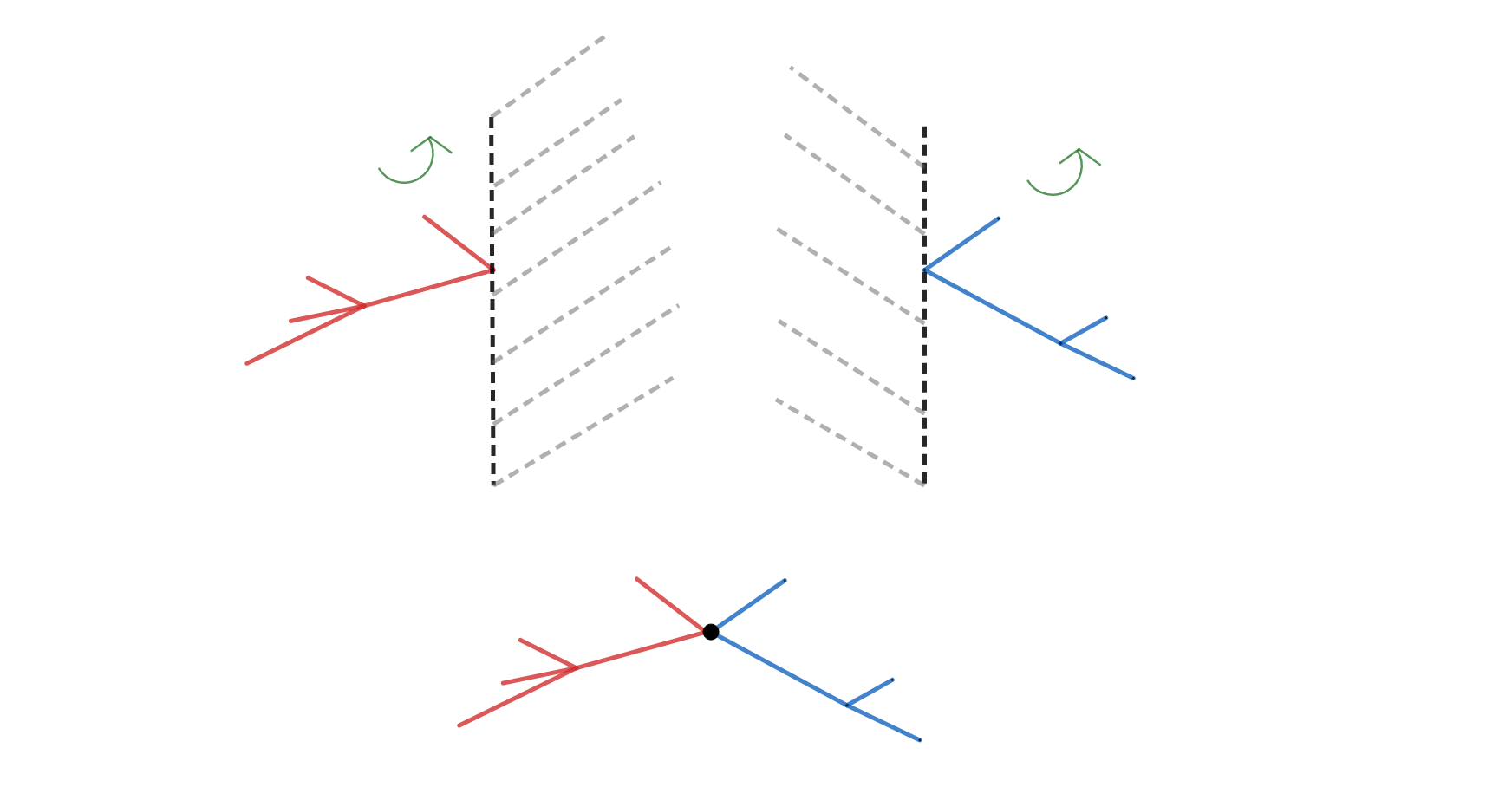} 
\captionof{figure}{The union tree of a negative tree and a positive tree.\label{fig:unionTree}}}
\vspace{10pt}
\end{example}

\begin{remark}
Since the plane is cooriented, the image of the origin by the quotient map $q$ separates the asymptotic Poincaré-Reeb tree in a negative tree and a positive tree. They form a union tree, whose root is the image of the origin (see Figure \ref{fig:unionTree}). 
\end{remark}

{\centering\vspace{10pt}
\includegraphics[scale=0.15]{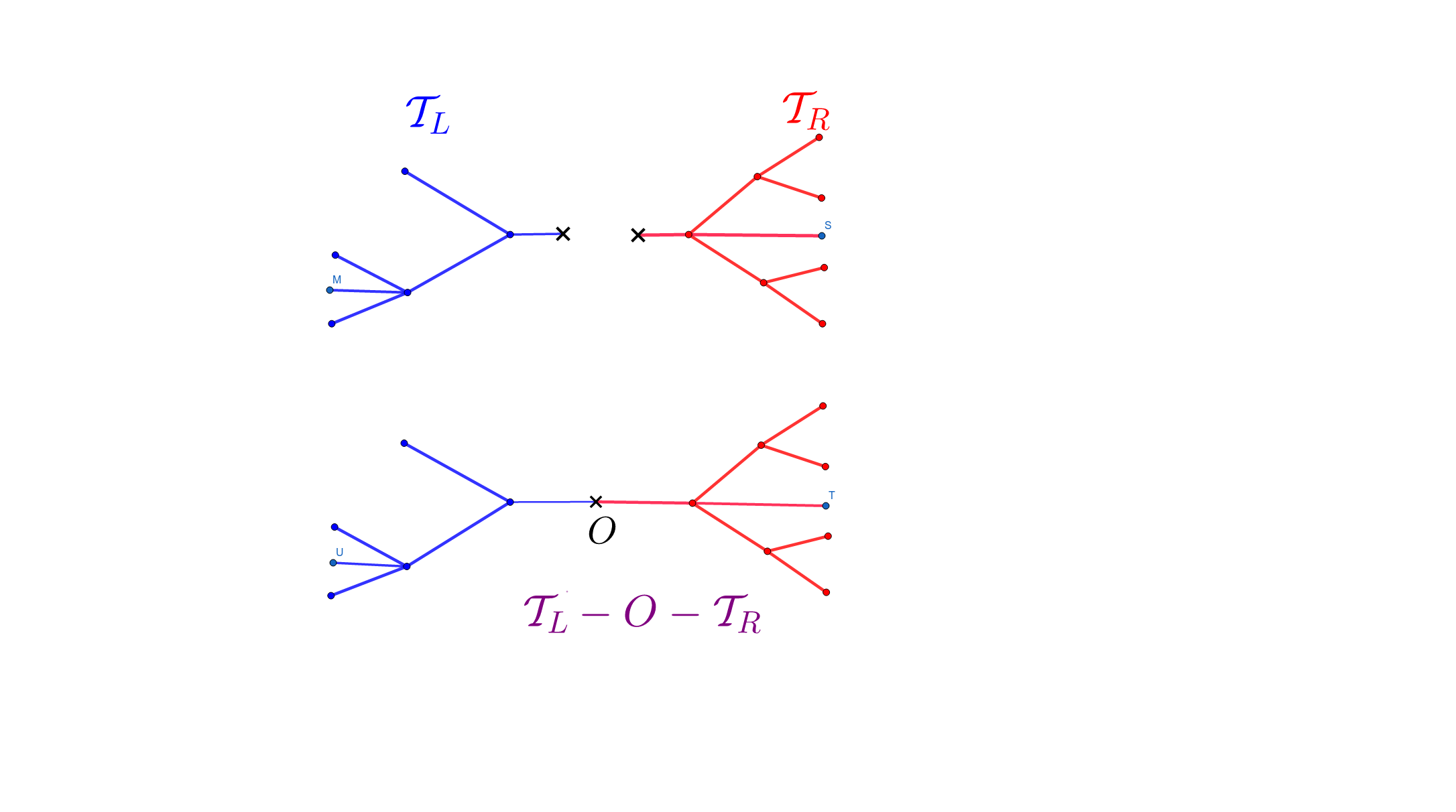} 
\captionof{figure}{The asymptotic Poincaré-Reeb tree is a union of a negative tree and a positive tree.\label{fig:unionasympt}}}
\vspace{10pt}

 \begin{remark}
 The root is a special vertex of the union tree. More precisely, the root of the union tree will be the only vertex with valency equal to two. 
 \end{remark}

\subsection{Crests and valleys}\label{sect:CrestValley}
Let us suppose from now on until the end of the paper that the polar curve $\Gamma(f,x)$ is reduced. This is equivalent to the fact that the levels of $f$ have only order $2$ intersections with their vertical tangents.

Let $\gamma^*$ be a polar half-branch (see \cite[Definition 5.13]{So3}) and let $(x_0,y_0):=\gamma^*\cap\mathcal{C}_\varepsilon.$ By Definition \ref{DefPolCurve} of the polar curve\index{polar curve} $\Gamma(f,x)$, the tangent to $\mathcal{C}_\varepsilon$ at the point $(x_0,y_0)$ is the vertical line $(x=x_0).$ In the following, we shall use the notations introduced in the previous sections.

\begin{definition}\label{DefCRestValley}
Let $\gamma^*$ be a polar half-branch. We say that a point $(x_0,y_0):=\gamma^*\cap \mathcal{C}_\varepsilon$ is a \defi{right crest}\index{crest!right crest} (respectively \defi{right valley}\index{valley!right valley}) of $\mathcal{C}_\varepsilon$ if and only if $x_0>0$ and $\mathcal{C}_\varepsilon$ is situated to the right (respectively, left) of the vertical line $(x=x_0),$ in a small enough neighbourhood of $(x_0,y_0).$ 

Let $\gamma^*$ be a polar half-branch. We say that a point $(x_0,y_0):=\gamma^*\cap \mathcal{C}_\varepsilon$ is a \defi{left crest}\index{crest!left crest} (respectively \defi{left valley}\index{valley!left valley}) of $\mathcal{C}_\varepsilon$ if and only if $x_0<0$ and $\mathcal{C}_\varepsilon$ is situated to the left (respectively, right) of the vertical line $(x=x_0),$ in a small enough neighbourhood of $(x_0,y_0).$ 

See Figure \ref{fig:rlCrestValley}.
\end{definition}

{\centering\vspace{10pt}
\includegraphics[scale=0.14]{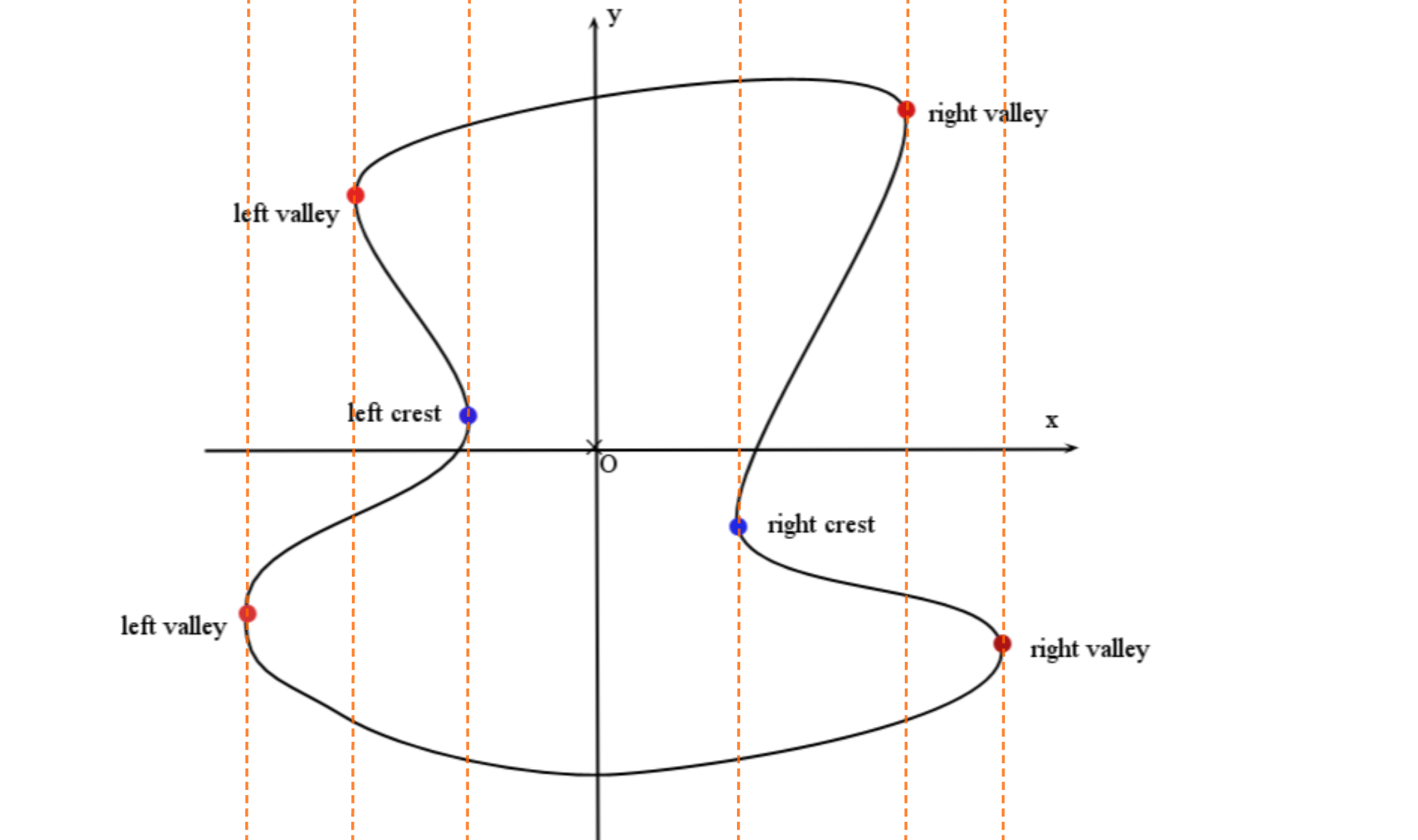}  
\captionof{figure}{Left and right crests in blue; left and right valleys in red.\label{fig:rlCrestValley}}}
\vspace{10pt}

\begin{remark}
The terminology of \textbf{crest} and \textbf{valley} is inspired by \cite[Definition 2]{CP}, where Coste and de la Puente used it to study atypical values of real bivariate polynomial functions at infinity. It is consistent with the one from geographical maps (see, for instance, \cite{geo1} and \cite{geo2}). However, in our case the observer who is studying the topographic map of the three-dimensional landscape is situated at infinity in the vertical direction of the $Oz$ axis. So our terminology of \enquote{valley} and \enquote{crest} is not intrinsic, since it depends on the position of the observer.
\end{remark}

\begin{remark}\label{remark:righcrest}
By Definition \ref{DefCRestValley}, one has the following properties in a small enough neighbourhood of $(x_0,y_0):$

(i) $(x_0,y_0)=\mathcal{C}_\varepsilon\cap\gamma^*$ is a right crest if and only if $x_0>0$ and $x_{|\mathcal{C}_\varepsilon}\geq x_0.$

(ii) $(x_0,y_0)=\mathcal{C}_\varepsilon\cap\gamma^*$ is a right valley if and only if $x_0>0$ and $x_{|\mathcal{C}_\varepsilon}\leq x_0.$

(iii) $(x_0,y_0)=\mathcal{C}_\varepsilon\cap\gamma^*$ is a left crest if and only if $x_0<0$ and $x_{|\mathcal{C}_\varepsilon}\leq x_0.$

(iv) $(x_0,y_0)=\mathcal{C}_\varepsilon\cap\gamma^*$ is a left valley if and only if $x_0<0$ and $x_{|\mathcal{C}_\varepsilon}\geq x_0.$

\end{remark}

Let us define two total order relations: one on the set of right polar half-branches and another one on the set of left polar half-branches. This is possible by Corollary \cite[Corollary 5.22]{So3}: in a good neighbourhood $V$ of the origin (see Definition \ref{def:GoodNeighb}), the polar half-branches intersect only at the origin.

\begin{definition}\label{DefOrderRel}
The \defi{canonical total order} \textbf{on the set of right polar half-branches} is the 
       restriction of the trigonometric order to it. The \defi{canonical total order on  the set of left polar half-branches} is the 
       restriction of the anti-trigonometric order to it.
\end{definition}

\begin{remark}
The order relation from Definition \ref{DefOrderRel} is well-defined, by Corollary \cite[Corollary 5.22]{So3}. 
Thus all the inequalities are strict and we have a total order relation between the right polar half-branches and a total order relation between  the left polar half-branches.
\end{remark}

{\centering\vspace{10pt}
\includegraphics[scale=0.2]{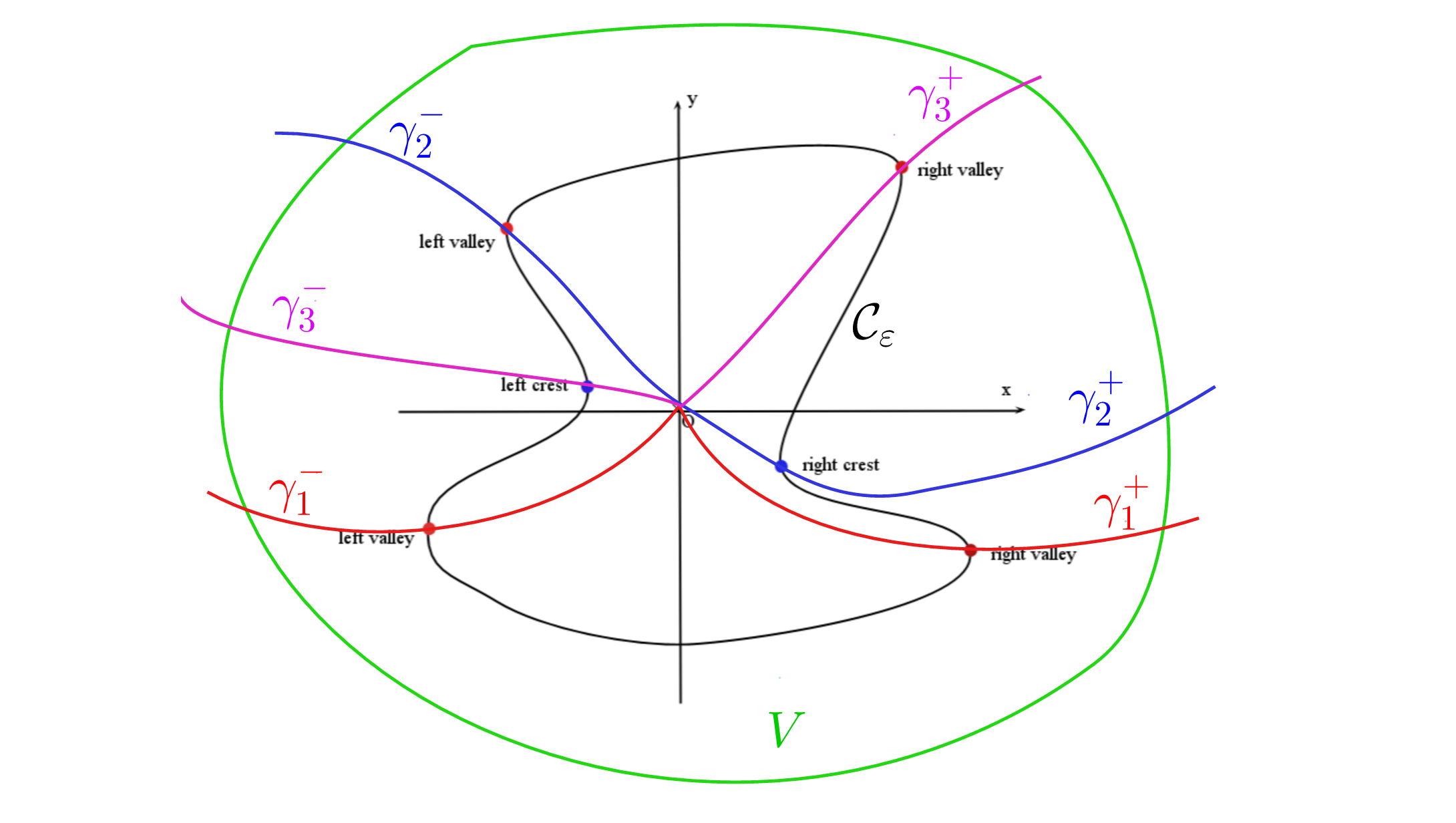} 
\captionof{figure}{Right half-branches: $\gamma_1^+<\gamma_2^+<\gamma_3^+$ and left half-branches: $\gamma_1^-<\gamma_3^-<\gamma_2^-$ of the polar curve $\Gamma(f,x)$.\label{fig:halfbranches}}}
\vspace{10pt}

\begin{proposition}\label{prop:alternating}
Let $\gamma_1^+<\gamma_2^+<\ldots<\gamma_k^+$ be all the right polar half-branches of $\Gamma(f,x)$. If $P_i:=\mathcal{C}_\varepsilon \cap \gamma_i^+,$ $P_i(x_i,y_i)$ in a good neighbourhood $V$ of the origin, then the points $P_1, P_2,\ldots,P_k$, in this order, consist of an alternation of right valleys and right crests, i.e. there are no two consecutive right valleys and no two consecutive right crests.

Similarly, to the left there is an alternation of left valleys and left crests, that is, there are no two consecutive left valleys and no two consecutive left crests.
\end{proposition}

\begin{proof}
It is sufficient to prove the statement only for the right side, and only for two consecutive right polar half-branches. 

We shall prove that there are no two consecutive right polar half-branches $\gamma_{i_0}$ and $\gamma_{i_0+1}$ such that the points $P_{i_0}(x_{i_0},y_{i_0}):=
\mathcal{C}_\varepsilon\cap
\gamma_{i_0}$ and $P_{i_0+1}(x_{i_0+1},y_{i_0+1}):=
\mathcal{C}_\varepsilon\cap
\gamma_{i_0+1}$ are both right valleys of $\mathcal{C}_\varepsilon.$ 

We argue by contradiction. Let us suppose the contrary, namely that $P_{i_0}$ and $P_{i_0+1}$ are both right valleys. Furthermore, assume without loss of generality that $x_{i_0+1}<x_{i_0}.$

Let us consider $a\in]x_{i_0+1},x_{i_0}[$ and the vertical line $x=a.$ Recall Definition \ref{def:GoodNeighb} and let us choose a good neighbourhood $V$ of the origin. Then there is a unique intersection point $A:=\gamma_{i_0+1}^+\cap (x=a)$ and a unique intersection point $B:=\gamma_{i_0}^+\cap (x=a)$. Hence one obtains the points $A(a,y_A)$ and $B(a,y_B).$ By Lemma \cite[Lemma 5.18]{So3}, since $f_{|\gamma_{i_0}^+}$ and $f_{|\gamma_{i_0+1}^+}$ are strictly increasing in $V$, we have $f(A)>f(P_{i_0+1})=\varepsilon$ and $f(B)<f(P_{i_0})=\varepsilon$. See Figure \ref{fig:altern}.

{\centering\vspace{10pt}
\includegraphics[scale=0.2]{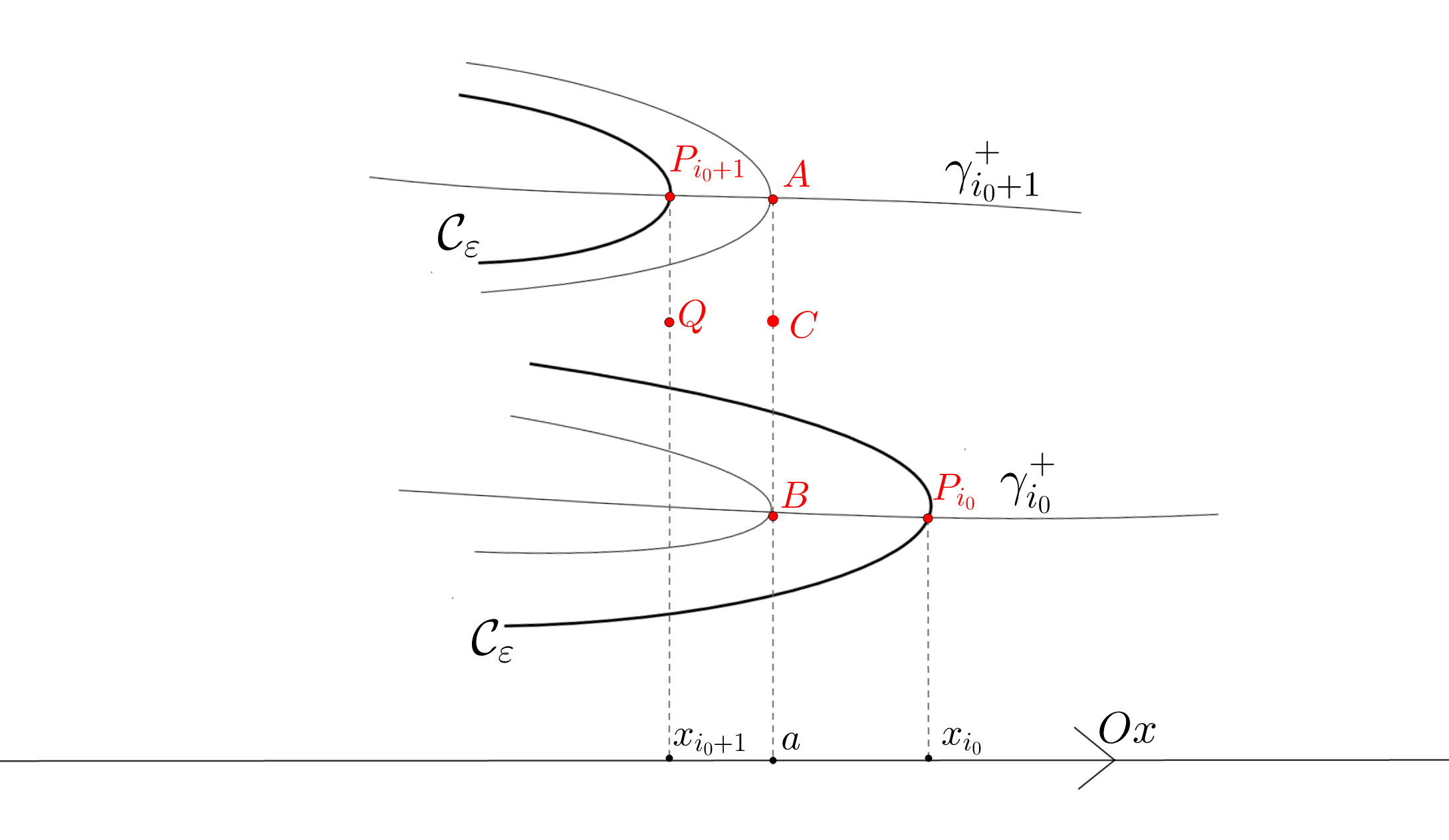} 
\captionof{figure}{No two consecutive right valleys.\label{fig:altern}}}

\vspace{10pt}

 If we denote by $F:\mathbb{R}\rightarrow\mathbb{R}$ the function $$F(y):=f(a,y)-\varepsilon,$$
then $F(y_A)>0$ and $F(y_B)<0.$ By the continuity of the polynomial function $F,$ there exists $b\in\mathbb{R}$, $y_B<b<y_A,$ such that $F(b)=0.$ Namely, there exists a point $C(a,b)$ with $y_B<b<y_A,$ such that $C(a,b)\in \mathcal{C}_\varepsilon.$ 

Now since the above holds for any $a\in]x_{i_0+1},x_{i_0}[,$ let us consider a continuous sequence $(a_k)\in ]x_{i_0+1},x_{i_0}[.$ We obtain a sequence of points $C_k(a_k,b_k)$ with $y_B<b_k<y_A,$ such that $C(a_k,b_k)\in \mathcal{C}_\varepsilon.$ Therefore we have just proved that for all $a_k\in]x_{i_0+1},x_{i_0}[$, there exists a point $C(a_k,b_k)\in\mathcal{C}_\varepsilon,$ i.e. $F(a_k,b_k)=0.$ Thus one obtains a continuous sequence of points $C(a_k,b_k)$ in the compact set $\mathcal{C}_\varepsilon$, hence one can subtract a convergent subsequence $C'(a_k,b_k)$ which has its limit in $\mathcal{C}_\varepsilon.$ Denote this limit by $L:=\lim_{a_k\rightarrow x_i}C'(a_k,b_k).$

Now there are two possibilities:
\begin{enumerate}
\item if $L\equiv P_{i_0+1},$ then one obtains a contradiction with Remark \ref{remark:righcrest}, (ii), namely with the definition of a valley at $P_{i_0+1}$;

\item if $L\equiv Q\neq P_{i_0+1}$, $Q(x_{i_0+1},y_Q)$:

let us first redefine the function $F:\mathbb{R}\rightarrow\mathbb{R}$, namely $$F(y):=f(x_{i_0+1},y)-\varepsilon;$$ since both $Q\in\mathcal{C}_\varepsilon$ and $P_{i_0+1}\in\mathcal{C}_\varepsilon$
, i.e. $F(y_Q)=F(y_{i_0+1})=0,$ one can apply Rolle's theorem to the continuous and differentiable function $F$. Hence there exists $y_S\in ]y_Q,y_{i_0+1}[$ such that $F'(y_S)=0.$ But $F'(y_S)=\frac{\partial f}{\partial y}(x_{i_0+1},y_s),$ thus the point $S(x_{i_0+1},y_s)\in\Gamma$. 
This is in contradiction with the hypothesis that $\gamma_{i_0}^+$ and
$\gamma_{i_0+1}^+$ are \textbf{consecutive} right polar half-branches.

\end{enumerate}
In conclusion, we have proved that two consecutive right polar half-branches cannot determine two right-valleys of $\mathcal{C}_\varepsilon.$

Using a similar reasoning, one can prove that two consecutive right polar half-branches cannot give two right-crests of $\mathcal{C}_\varepsilon.$

We conclude that the right valleys and the right crests must alternate. A similar result can be established for left valleys and left crests.
\end{proof}

\begin{corollary}
If $\gamma_i^+$ and $\gamma_{i+1}^+$ are two consecutive right polar half-branches with $P_i:=\mathcal{C}_\varepsilon\cap \gamma_i^+$ and $P_{i+1}:=\mathcal{C}_\varepsilon\cap \gamma_{i+1}^+$, $0<\varepsilon\ll 1,$ in a good neighbourhood 
$V$ of the origin (see Definition \ref{def:GoodNeighb}), then the following implications hold:
\begin{enumerate}
\item If $P_{i+1}$ is a right-crest and $P_i$ is a right-valley, then $x_{i+1}<x_i.$
\item If $P_{i+1}$ is a right-valley and $P_i$ is a right-crest, then $x_i<x_{i+1}.$
\end{enumerate} 
\end{corollary}

\begin{proof}

\begin{enumerate}

\item We argue by contradiction. If we suppose that $x_{i+1}>x_i$, then we apply the same steps as in Proposition \ref{prop:alternating} and by Rolle's theorem we obtain a contradiction with the hypothesis that $\gamma_{i}^+$ and $\gamma_{i+1}^+$ are consecutive right polar half-branches.
\item Analogous reasoning.
\end{enumerate}
\end{proof}

\begin{lemma}\label{lemma:lastValley}
The point $P_1:=\mathcal{C}_\varepsilon\cap \gamma_1^+,$ $P_1(x_1,y_1)$ determined by the first right polar half-branch $\gamma_1^+$ and the point
$P_k:=\mathcal{C}_\varepsilon\cap \gamma_k^+,$ $P_k(x_k,y_k)$ determined by the last right polar half-branch $\gamma_k^+$ are both right-valleys of $\mathcal{C}_\varepsilon,$ in a good neighbourhood of the origin, denoted by $V$.
\end{lemma}

\begin{proof}
We argue by contradiction. Let us suppose that $P_k$ is a right-crest of $\mathcal{C}_\varepsilon,$ as in Figure \ref{fig:lastVall}.

{\centering\vspace{10pt}
\includegraphics[scale=0.15]{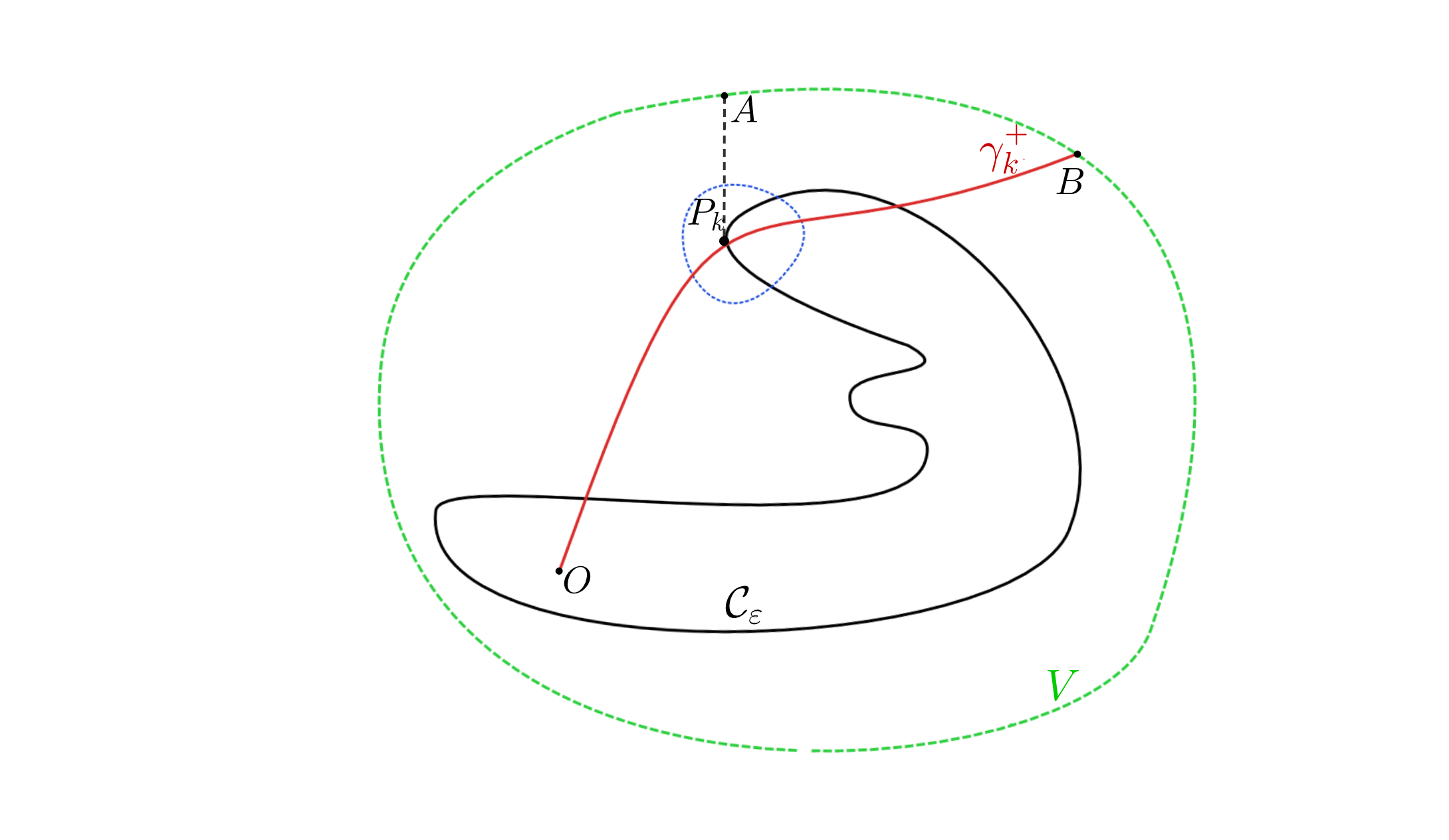} 
\captionof{figure}{Arguing by contradiction: if $P_k$ were a right crest.\label{fig:lastVall}}}
\vspace{10pt}

 Let us denote by $\delta:=\{(x,y)\in\mathcal{C}_
\varepsilon\mid x>x_k \text{ and } y>y_k\}\cap V$.

Let $A:=(x=x_k)\cap \partial V$ and $B:=\gamma_k^+ \cap \partial V.$ In the following we shall prove that $P_kAB$ is a \enquote{triangular} sector of $V$. This follows directly from Definition \ref{def:GoodNeighb} of a good neighbourhood, since the polar curve $\Gamma(f,x)$ has no vertical tangents in $V$. Namely, since $\Gamma(f,x)$ has no vertical tangents, neither the right polar half-branch $\gamma_k^+\subset\Gamma(f,x)$ has vertical tangents (see impossible Figure \ref{fig:contrad}).

{\centering\vspace{10pt}
\includegraphics[scale=0.15, trim={0 6cm 0 6cm}, clip]{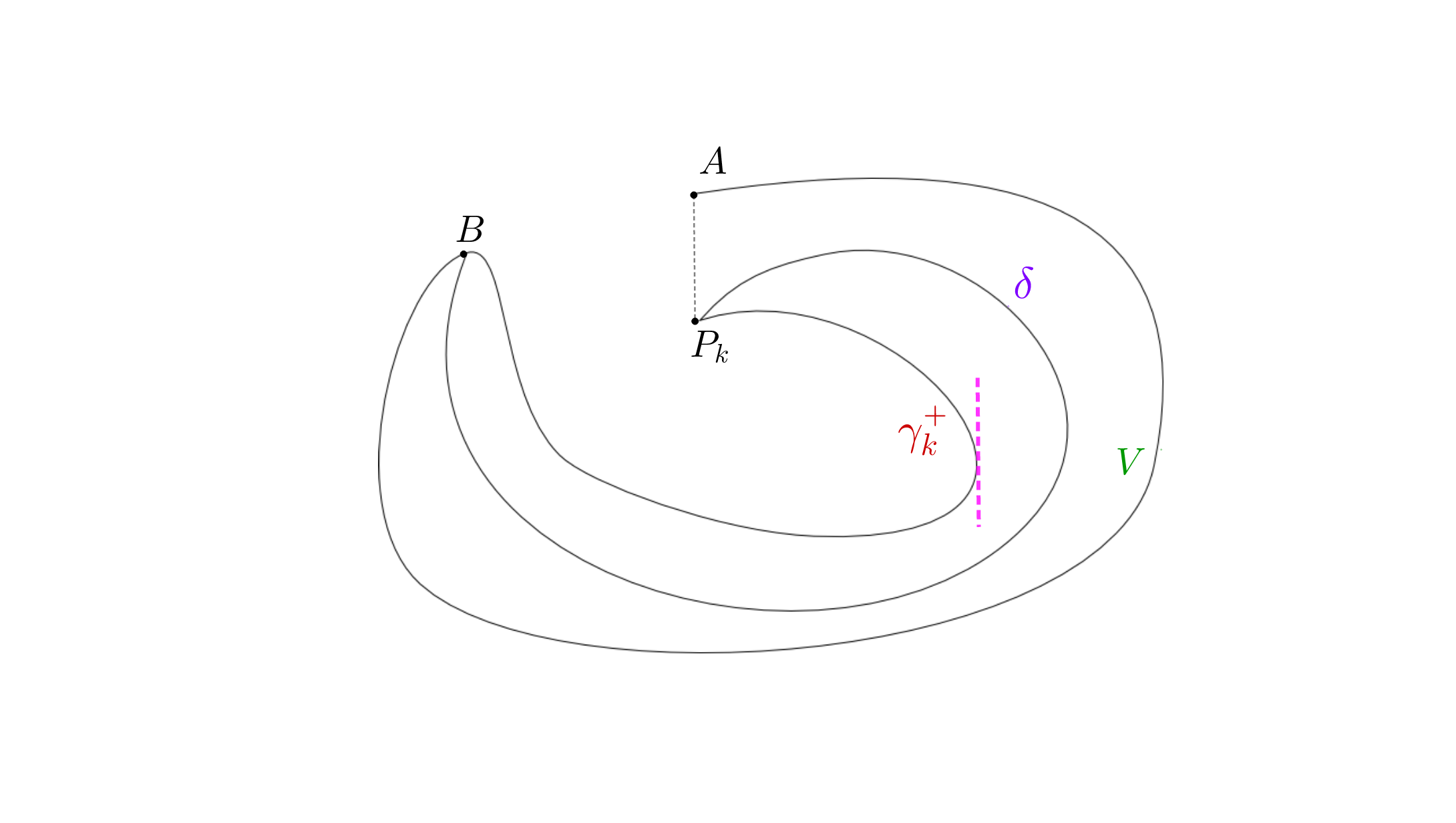} 
\captionof{figure}{No vertical tangents of $\gamma_k^+$ in $V$.\label{fig:contrad}}}
\vspace{10pt}

Moreover, since $\mathcal{C}_\varepsilon$ is a compact curve, $\delta$ has to \enquote{escape} from the triangular sector $P_kAB$. We want to prove that this leads us to a contradiction, in three steps, as follows.
\begin{enumerate}
\item We shall prove that $\delta \cap (x=x_k)=\emptyset$. We argue by contradiction. If there exists a point $Q(x_Q,y_Q)\in\mathcal{C}_\varepsilon$
with $x_Q=x_k$ and $y_Q>y_k$, then by Rolle's theorem one obtains another right polar half-branch $\gamma_{k+1}^+>\gamma_k^+.$ We have obtained a contradiction.

\item Since by hypothesis $\delta\subset \mathcal{C}_\varepsilon\subset V$, one has $\delta\cap \partial V=\emptyset.$
\item We prove now that $\delta\cap\gamma_k^+=\emptyset$. This follows from the fact that in a good neighbourhood $V$ of the origin, the right polar half-branch $\gamma_k^+$ intersects $\mathcal{C}_\varepsilon$ in precisely one point $P_k$, for $0<\varepsilon\ll 1.$
\end{enumerate}

We obtained a contradiction with the fact that $\mathcal{C}_\varepsilon$ has to escape from the triangular sector $P_kAB$, thus $P_k$ is a right- valley of $\mathcal{C}_\varepsilon.$

Suppose now that $P_1$ is a right-crest. We apply the same steps as above to prove that we obtain a contradiction, thus $P_1$ must also be a right-valley of $\mathcal{C}_\varepsilon.$
\end{proof}

In conclusion, to the right we have alternating right crests and right valleys, starting and ending with a right valley.
 Similarly, to the left.

\begin{lemma}\label{lemma:valleyLeaf}
Let $\gamma^+\subset \Gamma(f,x)$ be a right polar half-branch and let $P:=\gamma^+\cap\mathcal{C}_\varepsilon$, with $0<\varepsilon\ll 1$ be a right-valley of $\mathcal{C}_\varepsilon.$ Then the vertex corresponding to $P$ is a leaf of the Poincaré-Reeb tree, i.e. it has no children.
\end{lemma}

\begin{proof}
Let us consider a small enough neighbourhood of $P$, namely let us study $\mathcal{D}_\varepsilon$ locally around $P$.

By \cite[Lemma 5.3]{So3}, $\mathcal{C}_\varepsilon$ is a smooth Jordan curve. By applying the classical Jordan Curve Theorem (\cite[Lemma 5.2]{So3}), the set $\Int \mathcal{C}_\varepsilon:=
\{f<\varepsilon\}$ is either the bounded or the unbounded component of $\mathbb{R}^2\setminus\mathcal{C}_\varepsilon,$ as it is shown in Figure \ref{fig:2posib} below.

{\centering\vspace{10pt}
\includegraphics[scale=0.25, trim={6cm 7cm 7cm 4cm}, clip]{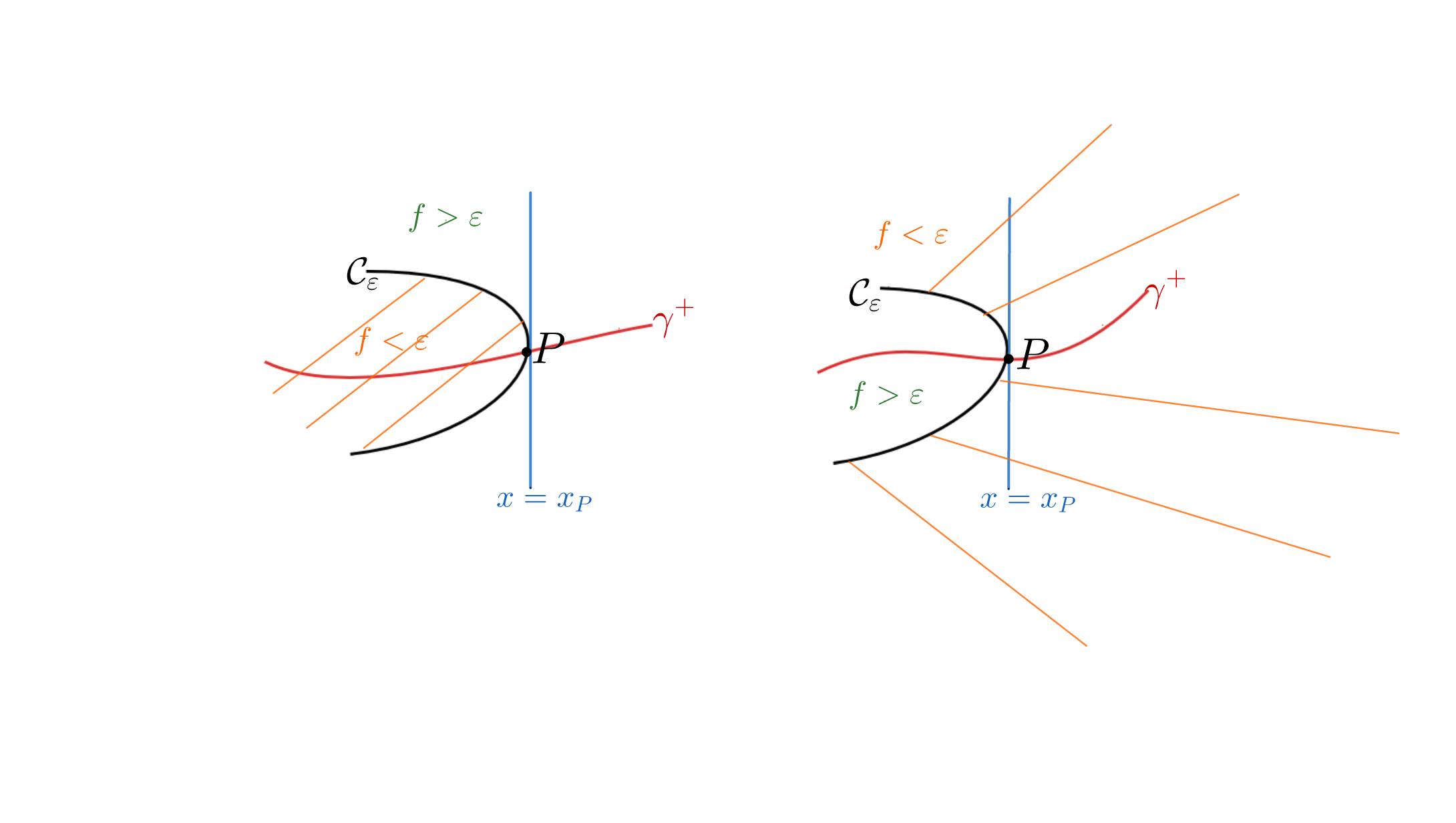} 
\captionof{figure}{Local study in a neighbourhood of a right valley of $\mathcal{C}_\varepsilon$.\label{fig:2posib}}} 
\vspace{10pt}

By \cite[Lemma 5.18]{So3}, one has $f_{|\gamma^+}$ is strictly increasing as one goes to the right, further from the origin. Hence, the only possible configuration is the one where  the line $x=x_P$ does not intersect any point of $\Int \mathcal{C}_\varepsilon$ (see Figure \ref{fig:2posib}, the left case). 

Let us now study the construction of the Poincaré-Reeb tree, locally, in a neighbourhood of a point $P$. After passing to the quotient, the image of the point $P$ is a vertex of the Poincaré-Reeb tree, say $P'\in \mathcal{R}(f,x),$ corresponding to the right-valley $P$. Since locally at $P$ there is only one connected component of $\mathcal{D}_\varepsilon,$ we obtain that the vertex $P'$ has no children, i.e. $P'$ is a leaf of $\mathcal{R}(f,x)$.
\end{proof}

\begin{remark}
A similar reasoning can prove that the vertices of $\mathcal{R}(f,x)$ which correspond to left valleys of $\mathcal{C}_\varepsilon$ are also leaves of $\mathcal{R}(f,x)$. Nevertheless, the vertices corresponding to left crests or to right crests are always internal vertices, provided that the direction of projection is generic.
\end{remark}

\begin{lemma}\label{lemma:CrestTwoChildren}
Let $\gamma^+$ be a right polar half-branch and let $P:=\gamma^+\cap\mathcal{C}_\varepsilon$, $0<\varepsilon\ll 1$, be a right-crest of $\mathcal{C}_\varepsilon.$
Let $\mathcal{R}(f,x)$ denote the Poincaré-Reeb tree. If the direction $x$ is a generic direction, then the vertex $P'\in \mathcal{R}(f,x)$ which is the image of $P$ by the quotient map, has exactly two children. In addition, both children will be situated to the right of $P'.$
\end{lemma}

\begin{proof}
By applying the same reasoning as in the proof of Lemma \ref{lemma:valleyLeaf}, one concludes that locally the line $x=x_P$ has no common point with $(f>\varepsilon)$.

Moreover, construct the Poincaré-Reeb tree locally in a neighbourhood of the point $P$. We obtain the vertex $P'$ corresponding to the right crest $P$ and we take into account the fact that to the left of the vertical line $x=x_P$ there is one connected component that we need to project, while to the right of the line $x=x_P$ there are two connected components to project. Therefore, $P'$ has two children and they are both oriented to the right of P. See Figure \ref{fig:twoChildren}.

{\centering\vspace{10pt}
\includegraphics[scale=0.15]{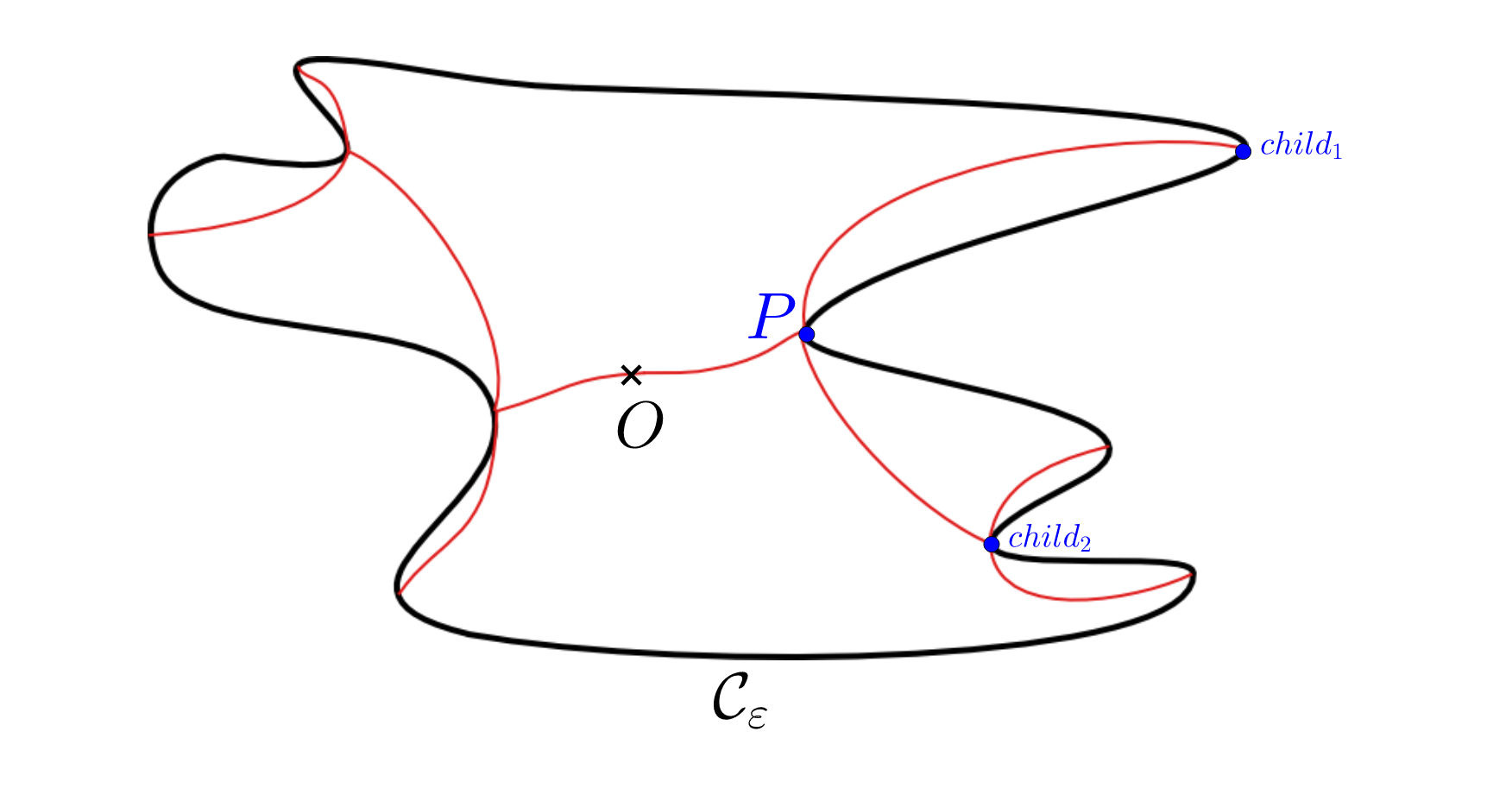} 
\captionof{figure}{The vertex corresponding to a right crest $P$ has two children to the right.\label{fig:twoChildren}}}
\end{proof}
\vspace{10pt}

\begin{remark}
Similar reasoning to prove that the vertices corresponding to left crests have two children, both oriented to the left.
\end{remark}

\subsection{Biordered sets}\label{subs:comparisonKnuth}

The set of graphs of a given finite family of polynomials in a small enough neighbourhood of a common zero can be endowed with two total orders (see \cite{Gh1}): one order given by their position for $x>0$ and the other order given by the position of the polynomials for $x<0$, as shown in Figure \ref{fig:GhysOrders}. Ghys proved (\cite[Theorem, page 31]{Gh1}) that a permutation can be realised by real polynomial graphs near a common zero if and only if it is a separable permutation. 

\newsavebox{\smlmat}
\savebox{\smlmat}{$\sigma:=\left(\begin{smallmatrix}
    1 & 2 & 3\\
    3 &1& 2
  \end{smallmatrix}\right)$}

{\centering\vspace{10pt}
\includegraphics[scale=0.1]{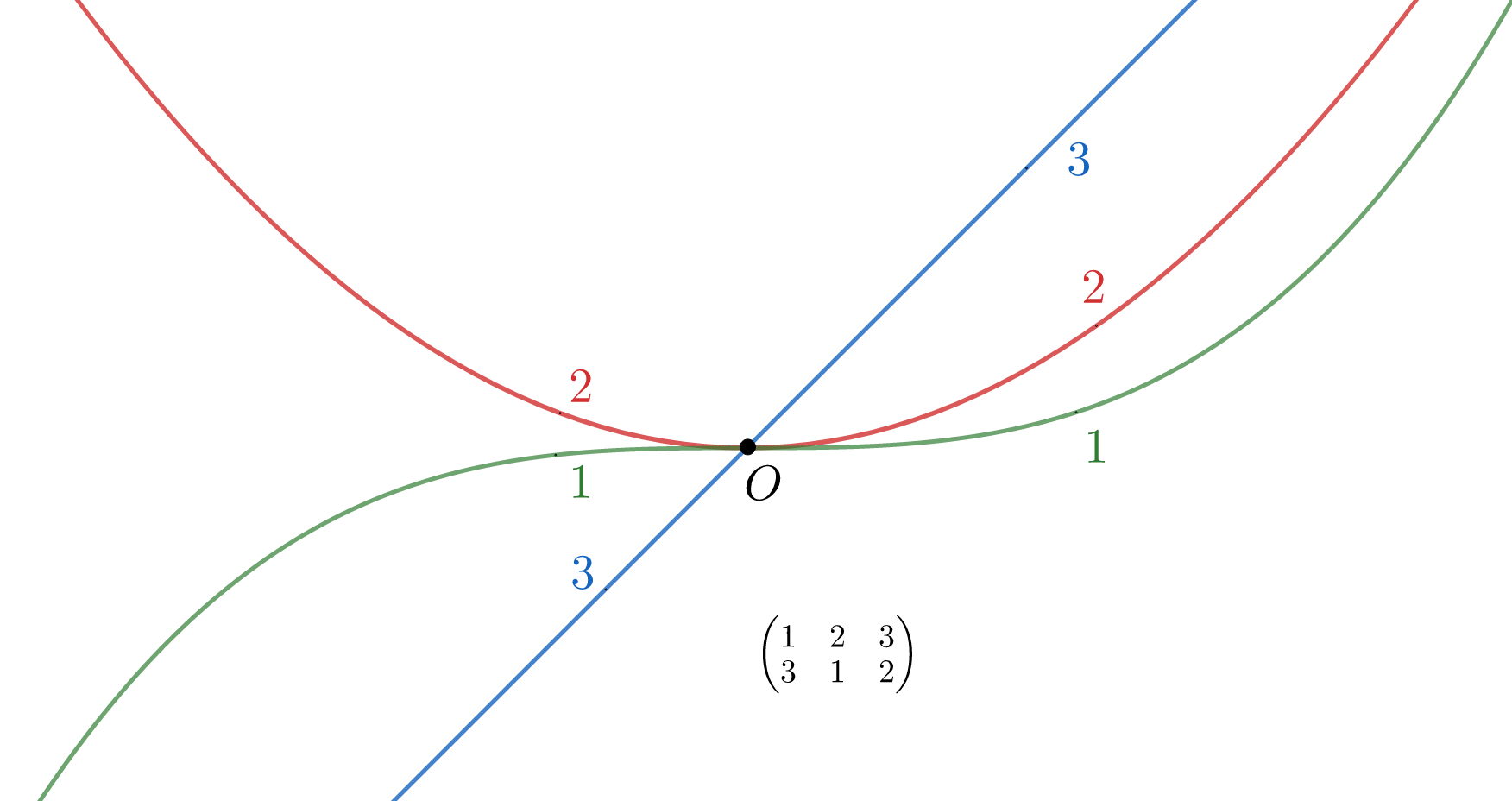} 
\captionof{figure}{Associating the permutation \usebox{\smlmat} to a set of graphs of polynomials near a common zero. For instance, the first polynomial in the order for $x>0$ arrives in the second position in the order for $x<0$. Thus in this case we say by convention that $\sigma(2)=1$. \label{fig:GhysOrders}}}
\vspace{10pt}

In the same manner, our goal in this section is to explain how two total order relations endowing the same set of objects allow us to define a permutation, as the one in Figure \ref{fig:GhysOrders}. This way of defining permutations is inspired from Knuth (see \cite[pages 17-18]{Gh1} for more details, and references therein: \cite{Kn} and \cite{Ki}). For us, the set of objects will be the right crests and right valleys.

Let us consider a finite set $\mathcal{A}$, with $n$ elements: $\#\mathcal{A}=n$. Let us take two total order relations on the set $\mathcal{A}$, denoted by $<_1$, respectively by $<_2.$ To such a pair we can associate an automorphism. The advantage of working with total order relations on the set $\mathcal{A}$ is that one can restrain the total order relations to subsets of $\mathcal{A},$ whereas one cannot restrain automorphisms of $\mathcal{A}$ to subsets of $\mathcal{A}.$

We use the notation $[n]:=\{1,2,\ldots,n\}$. Denote by $\mathrm{Aut(}\mathcal{A}\mathrm{)}$ (respectively $\mathrm{Aut(}[n]\mathrm{)}$) the set of automorphisms of $\mathcal{A}$ (respectively of $[n]$).

A total order relation on $\mathcal{A}$ can be seen as a bijection from the set $[n]$ to the set $\mathcal{A}$. Thus the function $<_1:[n]\rightarrow\mathcal{A}$ (respectively $<_2:\mathcal{A}\rightarrow[n]$) has an inverse function, namely ${<_1}^{-1}:\mathcal{A}\rightarrow[n]$ (respectively ${<_1}^{-1}:[n]\rightarrow\mathcal{A}$). Therefore, we have the following examples of elements of the set $\mathrm{Aut(}\mathcal{A}\mathrm{)}:$ the identity $\mathrm{id}$, $<_2\circ {<_2}^{-1}$ or $<_1\circ {<_2}^{-1}$. Similarly, we have the following examples of elements of the set $\mathrm{Aut(}[n]\mathrm{)}:$ the identity $\mathrm{id}$, ${<_2}^{-1}\circ <_1$ or ${<_1}^{-1}\circ <_2.$ Figure \ref{fig:totalOrdRelBij} illustrates the total order relations viewed as bijections.

\begin{figure}[H]
\centering
\begin{tikzpicture}
  \node (a) {$[n]$};
  \node[right=4cm of a] (b) {$\mathcal{A}$};
  \draw[->]
    (a) edge[bend left] node (f) [above]{$<_1$} (b)
    (a) edge[bend right] node (h) [below]{$<_2$} (b);
\end{tikzpicture}
\caption{Two total order relations viewed as bijections.\label{fig:totalOrdRelBij}}
\end{figure}
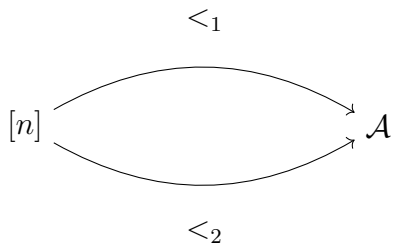

In addition, in Figure \ref{fig:totalOrdRelBijChoice1} and Figure \ref{fig:totalOrdRelBijChoice2} below we can see the two ways of constructing automorphisms, given two total order relations $<_1$ and $<_2$ on the same set.
\begin{figure}[H]
\centering
\begin{tikzpicture}
  \node (a) {$[n]$};
  \node[right=4cm of a] (b) {$[n]$};
  \draw[->]
    (a) edge[bend left] node (f) [above]{$\mathrm{id}$} (b)
    (a) edge[bend right] node (h) [below]{${<_2}^{-1}\circ <_1$} (b);
\end{tikzpicture}
\caption{Two automorphisms of the set $[n]$, which give us the permutation 
$i\mapsto{<_2}^{-1}\circ<_1(i)$.\label{fig:totalOrdRelBijChoice1}}
\end{figure}
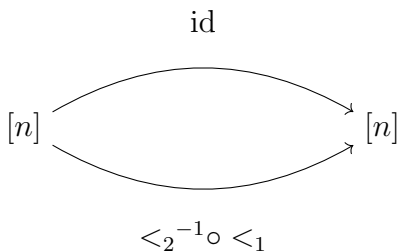

 \begin{figure}[H]
\centering
\begin{tikzpicture}
  \node (a) {$[n]$};
  \node[right=4cm of a] (b) {$[n]$};
  \draw[->]
    (a) edge[bend left] node (f) [above]{$\mathrm{id}$} (b)
    (a) edge[bend right] node (h) [below]{${<_1}^{-1}\circ <_2$} (b);
\end{tikzpicture}
\caption{Two automorphisms of the set $[n]$, which give us the permutation $i\mapsto {<_1}^{-1}\circ <_2(i)$.\label{fig:totalOrdRelBijChoice2}}
\end{figure}
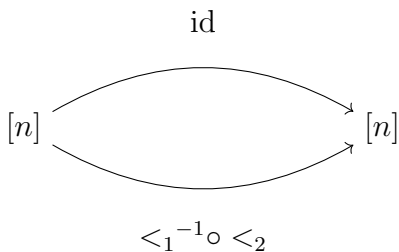

By convention, we shall choose the automorphisms given by the formula from Figure \ref{fig:totalOrdRelBijChoice1}, that is $i\mapsto{<_2}^{-1}\circ<_1(i)$. We can now define the notion of biordered set (see \cite{Gh1}), viewed from the point of view of the comparison between two total order relations on the same set.

\begin{definition}\cite[pages 17-18]{Gh1}\label{def:knuthPerm}
A \defi{biordered set}\index{biordered set} is a finite set, say $\mathcal{A},$ endowed with two total order relations. By convention, given a couple $\{<_1,<_2\}$ of total ordered relations, \defi{the Knuth automorphism of the biordered set}\index{biordered set!Knuth automorphism} $(\mathcal{A},<_1,<_2)$ is defined by  $\sigma(i):={<_2}^{-1}\circ<_1(i).$
\end{definition}
By abuse of language, we call a biordered set, a \textbf{permutation}, when there is no ambiguity.

In other words, let us denote by $a_i$ the element of $\mathcal{A}$ which appears in the position $i$ in the first total order relation. Then let us find the position $p$ of $a_i$ in the second total order relation. Thus Definition \ref{def:knuthPerm} says that $\sigma(i):=p$. 

\begin{example}
Given the set $\mathcal{A}:=\{a_1,a_2,a_3\}$ with the first total order relation $a_1<_1 a_2 <_1 a_3$ and the second total order relation $a_1 <_2 a_3 <_2 a_2$, we have (see Figure \ref{fig:comparatie}):

$${<_2}^{-1}\circ <_1 (1)={<_2}^{-1}(a_1)=1;$$
$${<_2}^{-1}\circ <_1 (2)={<_2}^{-1}(a_2)=3;$$
$${<_2}^{-1}\circ <_1 (3)={<_2}^{-1}(a_3)=2.$$

By Definition \ref{def:knuthPerm}, the Knuth automorphism of $(\mathcal{A},<_1,<_2)$ is $\sigma:=\begin{pmatrix}
    1 & 2 & 3\\
   1 &3& 2
  \end{pmatrix}.$

{\centering\vspace{10pt}
\includegraphics[scale=0.2, trim={0 12cm 0 6cm}, clip]{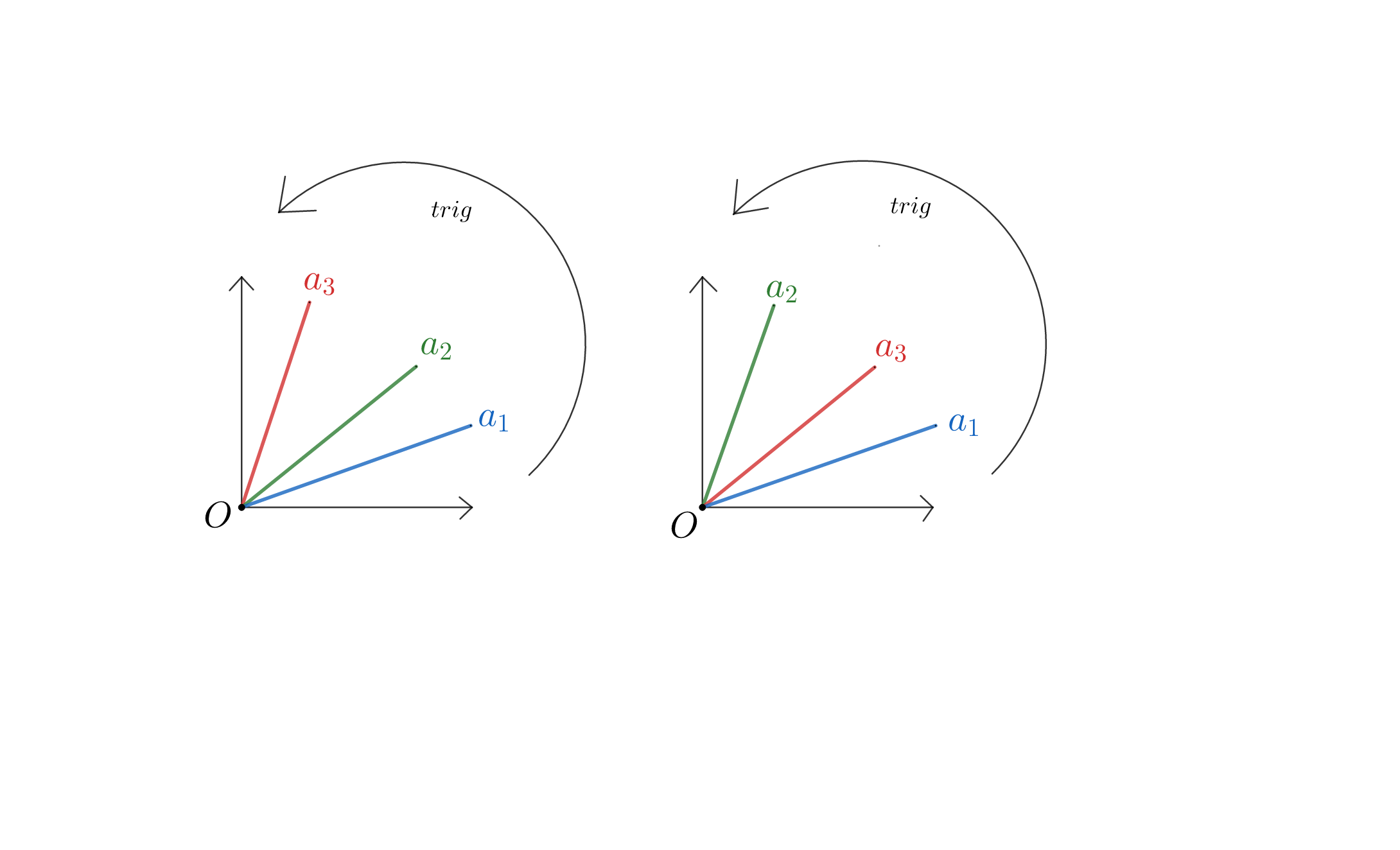} 
\captionof{figure}{The set $\mathcal{A}:=\{a_1,a_2,a_3\}$ endowed with two total order relations.\label{fig:comparatie}}}
\vspace{10pt}
\end{example}

\subsection{Snakes}
For the rest of this paper, let us restrict our attention only to what happens in the semi-plane $x>0$. Suppose that the direction of projection is generic.

We end this section with an application of generic projection directions. Namely, we present a new way to encode the shape of level curves, by what we call the asymptotic snake to the right of the origin. We are inspired by the notion of snake that Arnold introduced in the case of Morse univariate polynomials (see \cite[page 2]{Ar1}, \cite{CN}). The graphs of such polynomials were encoded by the so-called Arnold snakes (see for instance in \cite[Definition 1.8]{So2}). For a very recent construction of a large class of Arnold snakes see \cite[Theorem 5.1]{So2} (and \cite{So1}). 

Here we describe the snakes associated to the curves $\mathcal{C}_\varepsilon$, in the semi-plane $x>0$.

\begin{definition}\label{def:snake}
Let us consider a permutation $\sigma:\{1,2,\ldots,{n}\}\rightarrow \{1,2,\ldots,{n}\}.$ We say that $\sigma$ is \defi{a snake} if $[\sigma(1),\sigma(2),\ldots,\sigma(n)]$ verifies one of the following:
\begin{enumerate}
\item $\sigma(1)>\sigma(2)<\sigma(3)>\sigma(4)<\ldots;$
\item $\sigma(1)<\sigma(2)>\sigma(3)<\sigma(4)>\ldots;$
\item $\ldots  \sigma(n-2)>\sigma(n-1)<\sigma(n);$
\item  $\ldots  \sigma(n-2)<\sigma(n-1)>\sigma(n).$
\end{enumerate}
\end{definition}

\begin{theorem}\label{th:snakesBivar}
Let $f:\mathbb{R}^2\rightarrow\mathbb{R}$ be a polynomial function with a local strict minimum at the origin such that $f(0,0)=0.$ Consider a curve $\mathcal{C}_\varepsilon$ as in Definition \ref{def:levelCEpsilon}, for a sufficiently small $\varepsilon>0$. If the vertical projection is generic with respect to $\mathcal{C}_\varepsilon,$ then there exists a snake $\sigma:\{1,2,\ldots,{n}\}\rightarrow \{1,2,\ldots,{n}\}$, with $\sigma(1)<\sigma(2)>\sigma(3)<\ldots >\sigma(n)$ that encodes the shape of $\mathcal{C}_\varepsilon,$ in the semi-plane $x>0$. Here $n$ denotes the number of right polar half-branches. 
\end{theorem}

\begin{proof}
By \cite[Theorem 5.31]{So3}, the shape stabilises near the origin, for sufficiently small $\varepsilon>0$.

Since by hypothesis the vertical projection is generic with respect to $\mathcal{C}_\varepsilon$, the curve consists of alternating crests and valleys (see Proposition \ref{prop:alternating}).

Endow the set of right crests and valleys with two total order relations and get a permutation,  say $\sigma$ (Definition \ref{def:knuthPerm}). The first order is the one induced by the total order of the right polar half-branches. The second order is given by the vertical foliation induced by $x$. See Figure \ref{fig:twoTotalOrders}.

Since the crests and valleys alternate, the permutation $\sigma$ is a snake. In particular, by Lemma \ref{lemma:lastValley}, we have $\sigma(1)<\sigma(2)>\sigma(3)<\ldots >\sigma(n)$, since the first and last right polar half-branches correspond to right-valleys. 

\end{proof}

\begin{example}
See Figure \ref{fig:twoTotalOrders}. The first order is obtained by reading the right crests and right valleys along the curve (i.e. in the order given by the right polar half-branches), and the other order is given by the $x$-coordinate.

{\centering\vspace{10pt}
\includegraphics[scale=0.2, trim={0cm 7cm 0 6cm}, clip]{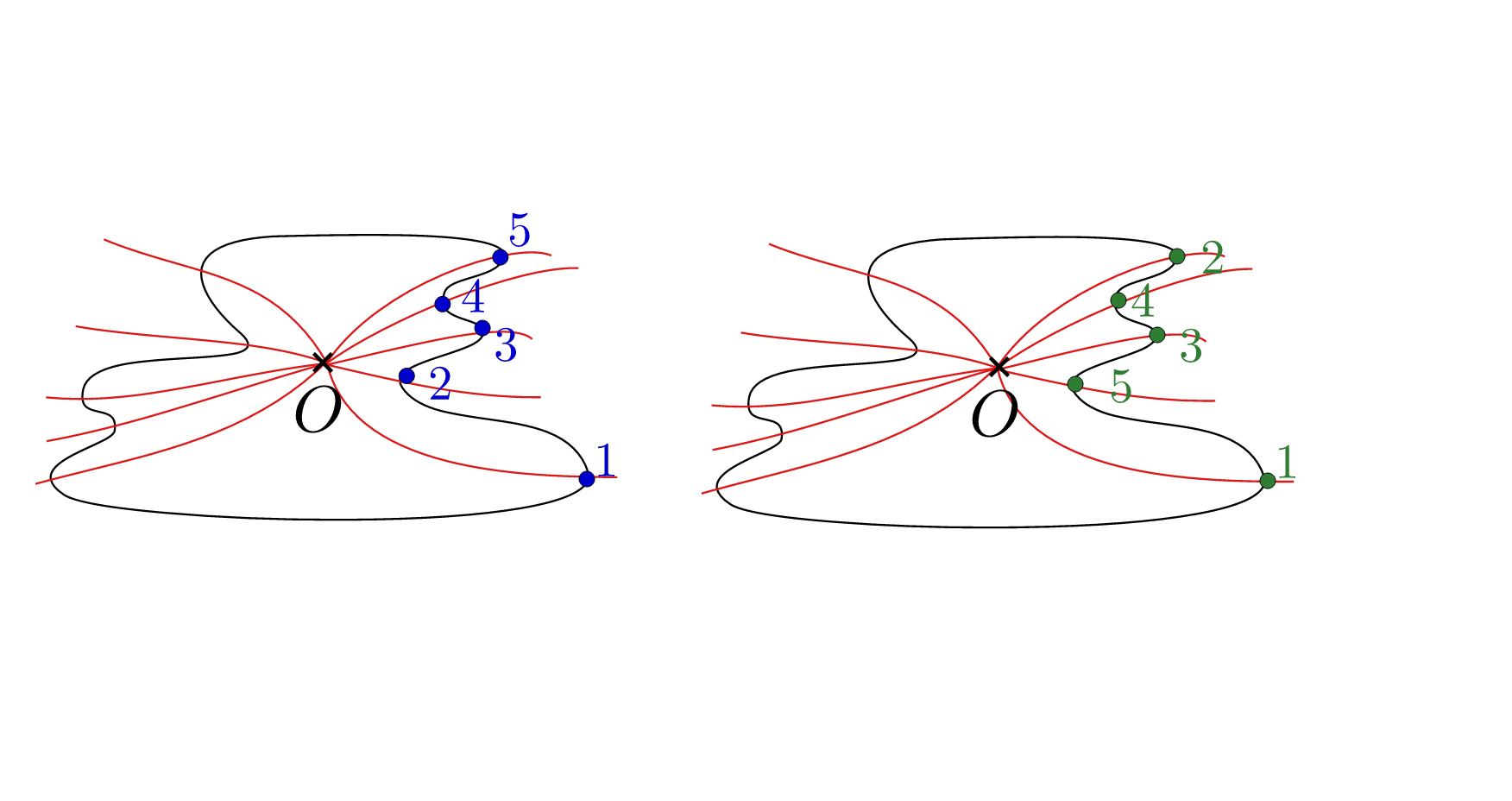} 
\captionof{figure}{Two total order relations on the set of right crests and right valleys.\label{fig:twoTotalOrders}}}
\vspace{10pt}

It is a biordered set of five points, to which we associate the bivariate snake $\sigma=\begin{pmatrix}
1 &2& 3&4&5\\
1 & 5 & 3&4 &2
\end{pmatrix},$ by choosing the convention of reading the green points from right to left and the convention from Definition \ref{def:knuthPerm}. See Figure \ref{fig:snakeRight}.

{\centering\vspace{10pt}
\includegraphics[scale=0.2, trim={0cm 0cm 0 0cm}, clip]{snakeToTheRight1.png} 
\captionof{figure}{The snake to the right of the curve $\mathcal{C}_\varepsilon$.\label{fig:snakeRight}}}
\vspace{10pt}

\end{example}

\begin{remark}
Note that by the construction of the Poincaré-Reeb tree, the order given by the $y-$coordinate may not always  correspond to the order along the level curve, see an example in Figure \ref{fig:exceptie}.

{\centering\vspace{10pt}
\includegraphics[scale=0.15]{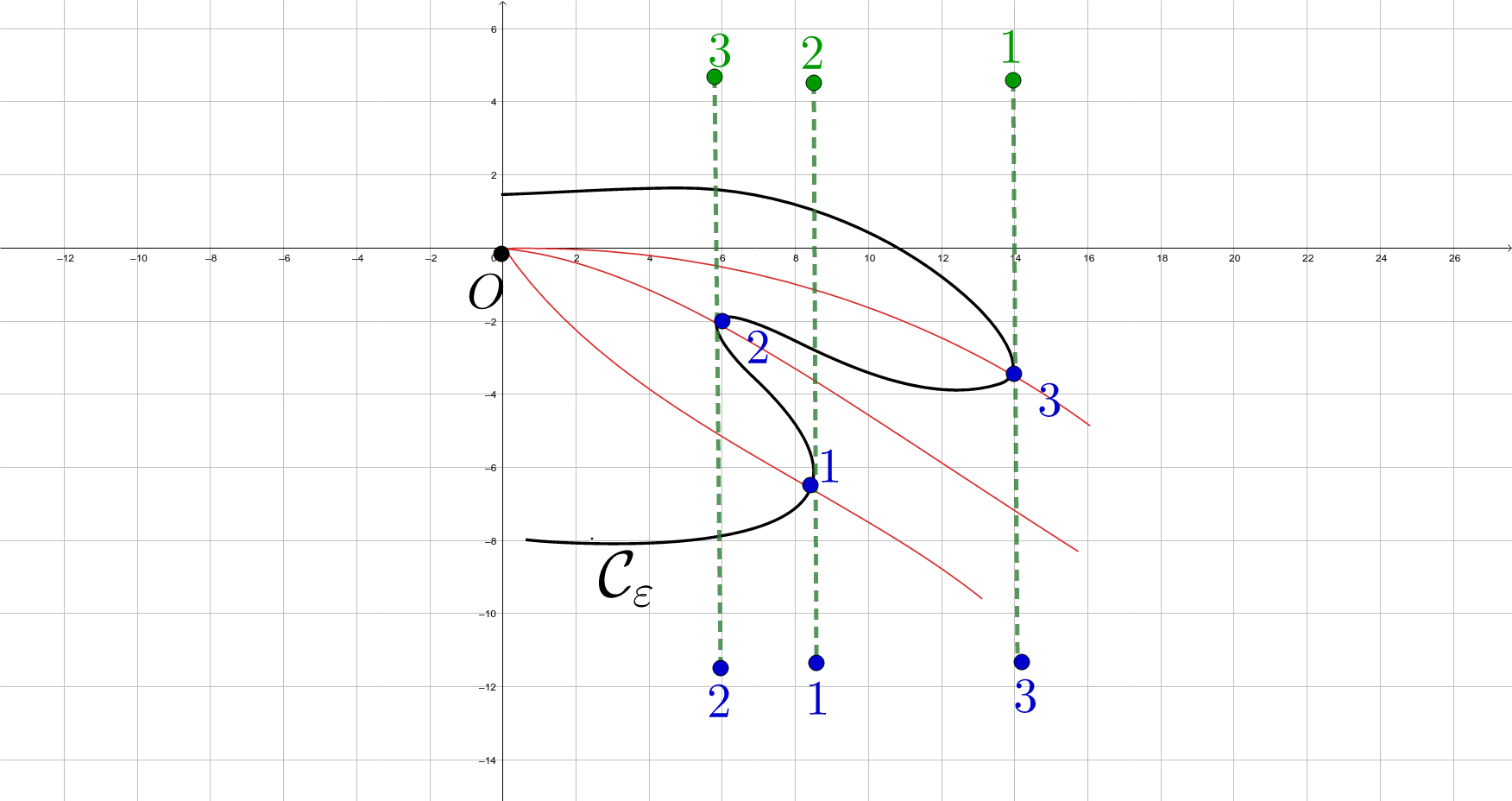} 
\captionof{figure}{The canonical order is the one obtained by reading the crests and valleys along the curve, not by their $y-$coordinate.\label{fig:exceptie}}}
\end{remark}
\vspace{10pt}

\printbibliography

\end{document}